\documentclass[hidelinks,onefignum,onetabnum]{siamart251216}

\usepackage{lipsum}
\usepackage{amsfonts}
\usepackage{xcolor}
\usepackage{graphicx}
\usepackage{epstopdf}
\usepackage{algorithmic}
\ifpdf
  \DeclareGraphicsExtensions{.eps,.pdf,.png,.jpg}
\else
  \DeclareGraphicsExtensions{.eps}
\fi

\newsiamremark{remark}{Remark}
\newsiamremark{hypothesis}{Hypothesis}
\crefname{hypothesis}{Hypothesis}{Hypotheses}
\newsiamthm{claim}{Claim}
\newsiamremark{fact}{Fact}
\crefname{fact}{Fact}{Facts}

\headers{Balanced Active Subspaces for Optimal Control}{Emmanuel S. Ameh, Samuel E. Otto}

\title{Balanced Active Subspaces for Optimal Control \\
via Trajectory Optimization%
\thanks{\funding{This work was funded by the Department of Mechanical and Aerospace Engineering, Cornell University.}}}

\author{Emmanuel S. Ameh\thanks{Department of Chemical and Biomolecular Engineering
  (\email{esa62@cornell.edu}, \url{https://amehhh.github.io/}).}
\and Samuel E. Otto\thanks{Department of Mechanical and Aerospace Engineering
  (\email{seo34@cornell.edu}, \url{https://sites.google.com/view/ottolab/home}).}}

\usepackage{amsmath}
\usepackage{amssymb}

\DeclareMathOperator{\Tr}{Tr}
\DeclareMathOperator{\rank}{rank}

\DeclareMathOperator{\diag}{diag}

\DeclareMathOperator{\Range}{Range}

\DeclareMathOperator{\grad}{\nabla}

\DeclareMathOperator*{\E}{\mathbb{E}}

\newcommand{\R}{\mathbb{R}}

\usepackage{tikz}
\usetikzlibrary{arrows.meta, positioning}
\usetikzlibrary{calc}
\usepackage{float}
\usepackage{placeins} % provides \FloatBarrier to keep floats within sections
\usepackage{mathtools}
\usepackage[caption=false]{subfig}
\usepackage[percent]{overpic}

\ifpdf
\hypersetup{
  pdftitle={Data-Driven Trajectory-Optimization-Based Model Reduction via Balanced States and Costates for Optimal Control of Fluid Flows},
  pdfauthor={A.E.Sunday, Samuel Otto}
}
\fi

\begin{document}

\maketitle

% REQUIRED
\begin{abstract}
The curse of dimensionality associated with solving the Hamilton–Jacobi–Bellman equation for the value function remains a key obstacle for real-time optimal feedback control of high-dimensional nonlinear dynamical systems. This is especially severe in aerospace applications such as relaminarizing unsteady turbulent flows and recovering from aerodynamic stall where nonlinear dynamics play a significant role. We introduce Balanced Active Subspaces for Optimal Control (BASOC), a model reduction framework which seeks to find directions in state space along which the value function is most sensitive, and the distribution of states exhibits maximum variance along optimal trajectories. BASOC leverages trajectory optimization and the Pontryagin minimum principle to obtain state and co-state (value  function gradient) data along locally optimal trajectories for constructing state and gradient covariance matrices. Balancing these matrices provides an oblique projection for approximating the value function, which can also be used to construct reduced-order models and surrogate feedback control laws. We show that for linear-quadratic optimal control problems, BASOC becomes equivalent to a snapshot method for initial condition balancing of the optimally controlled closed-loop realization. The framework is demonstrated on a nonlinear three-dimensional system where POD and balanced truncation are known to perform poorly, and on the nonlinear complex Ginzburg–Landau equation to demonstrate feasibility for larger-scale systems.
\end{abstract}

% REQUIRED

% REQUIRED

\section{Introduction}
Optimal control seeks to modify the behavior of a dynamical system by finding inputs that minimize a cost function measuring the quality of the system's behavior.
The optimal control problem leads to the Hamilton-Jacobi-Bellman (HJB) equation whose solution---the value function---provides the optimal cost-to-go from every state of the dynamical system. The value function is a key object as it allows the optimization problem over the entire input sequence to be converted into a simpler feedback law.
However, the extreme computational cost of representing, computing, and storing accurate approximations of the value function obstructs applications of optimal control and reinforcement learning to high-dimensional nonlinear systems such as turbulent separated flows governing aerodynamic stall recovery and flow relaminarization in aerospace applications \cite{Gad-el-Hak_2000,AamoKrstic2003}. This motivates the development of methods that can uncover and exploit low-dimensional structure in the HJB value function without computing it directly.

For continuous-time, infinite-horizon linear-quadratic regulator (LQR) optimal control problems involving linear dynamics, it is well known that the value function from any state configuration $x \in \mathbb{R}^n$ admits a closed-form expression $V(x) = x^T\Pi x$ \cite{AndersonMoore1990} and is characterized by $\Pi \succeq 0$, the stabilizing solution to the algebraic Riccati equation. Under assumptions such as a small number of inputs $m$ and outputs $p$ relative to the state dimension $n$, in addition to the spectrum of the controlled system matrix being far from its reflection about the imaginary axis, $\Pi$ has rapidly decaying singular values \cite{Stillfjord2018,benner2020comparison}. As a result, the value function for LQR problems can be accurately approximated as a function of a small number of active variables in a low-dimensional subspace (i.e., in an active subspace) \cite{Zahm2020,bigoni2022nonlinear}. This observation motivates the two fundamental research questions that we begin to explore in this paper: Do low-dimensional active subspaces and variables exist for the value function of nonlinear shear-dominated fluid flow problems, and if so, can they be exploited for control? 

To address these questions, we introduce the Balanced Active Subspaces for Optimal Control (BASOC) framework. BASOC seeks to identify an oblique projection operator for dimensionality reduction of the time-independent value function (cost-to-go) that jointly captures directions in state space along which the value function is most sensitive, and the distribution of states exhibits large variance. Our approach builds on the active subspace framework, which applies proper orthogonal decomposition (POD) to gradients of a scalar-valued function sampled according to a user-prescribed probability distribution to reveal an active subspace \cite{Lukaczyk2014}, i.e., a subspace spanned by directions along which the function predominantly varies. The active subspace method has been employed in numerous challenging dimensionality reduction and optimal design problems in engineering and science. Zahm et al. \cite{Zahm2020} generalize the active subspace method to multivariate vector-valued functions and identify a projection that balances sensitivity information against the variance of points in the domain of interest. The connection with balancing transformations was identified in \cite{Otto2023model} and applied to dimensionality reduction of large-scale dynamical systems by computing balanced active state variables for nonlinear observability operators, yielding a snapshot-based method called Covariance Balancing Reduction using Adjoint Snapshots (CoBRAS). For linear systems forced by white noise, the state and gradient covariances recover the controllability and observability Gramians, and CoBRAS recovers the Balanced Proper Orthogonal Decomposition (BPOD) \cite{Rowley2005} approximation of Balanced Truncation (BT) \cite{moore1981pca}. 

BASOC exploits the fact that the costate variables provided by the Pontryagin minimum principle (PMP) \cite{Pontryagin1962} are generalized gradients of the HJB value function along optimal trajectories \cite{ClarkeVinter1987,CannarsaFrankowska2017}. This connection between the HJB equation and PMP enables us to employ an indirect forward-backward sweep method (FBSM) of trajectory optimization to obtain state and gradient samples of the optimal value function along optimal trajectories, which BASOC then uses to compute balanced active variables and projections similarly to \cite{Otto2023model}. We note that in practice FBSM only computes locally optimal trajectories, so the gradient samples we obtain are actually from a state-constrained optimal value function. We explain how the resulting BASOC variables and oblique projections can be used to build surrogate control laws for full-state feedback control and projection-based ROMs of the dynamical system. For LQR problems, we show that BASOC provides a snapshot-based method for initial condition balancing \cite{Beattie2017} of the optimally controlled system. We evaluate BASOC-based reductions for value function approximation, model reduction, and downstream control tasks on a three-dimensional nonlinear system where POD and BT are known to struggle, and on the nonlinear complex Ginzburg--Landau equation.
We compare to reductions based on POD, BT, and another new reduction based on POD applied directly to optimal costates, which we call the POD-Adjoint method. Among these methods, only BASOC and BT respect arbitrary linear state transformations owing to their constructions via balancing.

\subsection{Related work}
Several other reduced basis methods have been introduced to address the computational burden of solving HJB equations. A related approach \cite{DolgoviKaliseSaluzzi2023} proposes a statistical POD method along optimal state trajectories for semilinear systems with random samples of initial conditions. The resulting reduced optimal feedback law is approximated in the POD subspace using tensor trains and evaluated cheaply online for near-optimal control. Although effective in this setting, the POD basis captures dominant state variance in the state space and approximates states, but not the value function, which may have steep gradients in directions that are not explored by the dominant state trajectories. A related alternative is to employ POD to obtain a ROM of the dynamical system via a Galerkin projection. By significantly reducing the degrees of freedom, the resulting HJB equation becomes computationally tractable \cite{KunischVolkweinXie2004}. However, this method can fail to capture low-variance, yet dynamically important (highly observable) states that are often important for modeling the dynamics of shear-dominated fluid flows such as boundary layers, shear layers, and wakes \cite{chomaz2005global, trefethen1993hydrodynamic, Schmid2001}. Near an equilibrium this problem can be addressed by applying oblique projection operators such as those obtained from balanced truncation or its snapshot-based approximation using BPOD \cite{moore1981pca, Rowley2005, Ilak2008modeling}. For nonlinear systems, a theoretical framework for balanced truncation around a stable equilibrium was developed by Scherpen \cite{Scherpen1993, fujimoto2010balanced}. While progress has been made towards scalable nonlinear balancing through tensor-based computations \cite{kramer2024scalable}, the same curse of dimensionality still arises unless appropriate low-dimensional structures can be identified and exploited. BASOC provides a tool to reveal a simple kind of low-dimensional active variable structure in the value function that could potentially also be applied to the energy functions in nonlinear balancing. Alla et al. \cite{Alla2017HJBMOR} review model reduction methods for HJB equations and suggest constructing a ROM via a POD basis generated from adjoint-equation simulations derived from optimality conditions. However, the adjoint equation is computed along non-optimal trajectories generated from random initial conditions and control inputs. Using optimal trajectories would recover our POD Adjoint method, which computes an active subspace of the value function. Adjoint POD presents the additional challenge that states are not necessarily well-approximated in the span of the resulting basis, making the approach poorly suited to constructing ROMs of dynamics and reduced-order observers.

Beyond reduced-basis methods, a variety of numerical approaches exploit alternative forms of structure to reduce the computational burden of solving HJB equations. One approach \cite{DolgovKaliseKunisch2021} suggests using low-rank tensor decomposition for the solutions of HJB PDEs. Deep learning approaches \cite{HanJentzenE2018} reformulate the HJB equation as a backward stochastic differential equation and employ a deep-learning-based method for approximating its solution. Another direction \cite{DarbonOsher2016} \cite{ChowDarbonOsherYin2019} further employs a grid-free algorithm for solving certain types of HJB equations with convex and state-dependent Hamiltonians based on the generalized Hopf--Lax formula.

\section{Preliminaries on Optimal Control} 

We consider the Lagrange form of the infinite-horizon optimal control problem in continuous time. The control problem is to find a control function $u(\cdot)$ that minimizes the undiscounted infinite-horizon cost functional
\begin{equation}
     J_{x_0}(u(\cdot)) := \int_{0}^{\infty} \mathcal{L}(x(t), u(t))\, dt,
     \label{eq:Lagrange_cost_functional}
\end{equation}
subject to the dynamics
\begin{equation}
    \begin{cases}
        \dot{x}(t) = f(x(t) , u(t)), & \forall t \in [0, \infty), \\
        y(t) = h(x(t),u(t)), & \forall t \in [0, \infty),  \\
        x(0) = x_{0}.
    \end{cases}
    \label{eq:Continous_time_dynamics}
\end{equation}
The evolution of the state $x(t) \in \mathbb{R}^n$ is governed by the nonlinear dynamics $f: \mathbb{R}^n  \times U\to \mathbb{R}^n$, where $U \subseteq \mathbb{R}^m$ denotes the set of admissible control values, with admissible control functions $ \mathcal{U}_{\mathrm{ad}} = \left\{ u:[0,\infty)\rightarrow U \;\middle|\; \mbox{$u$ measurable} \right\}$ and $\mathcal{L}:\mathbb{R}^n\times U\rightarrow\mathbb{R}$ measuring the running cost.
We envision the situation where the initial condition $x_0$ is an undesirable state, such as stalled flow over an aircraft wing, from which we would like the system to optimally recover in real time.
While it does not participate in the stated optimal control problem, our system description includes an output $y(t) \in \mathbb{R}^p$ given by a map $h: \mathbb{R}^n \times \mathbb{R}^m \to \mathbb{R}^p$ that we use later to design output-based compensators.

We assume that $f$ is continuous in $(x,u)$ and locally Lipschitz in $x$, uniformly with respect to $u\in U$, and that the resulting controlled trajectories exist on $[0,\infty)$
\cite{Bardi1997}. For initial states with finite optimal
cost, the value function is defined by
\begin{equation}
    V(x)=\inf_{u\in\mathcal U_{\mathrm{ad}}}
    J_x(u(\cdot)).
    \label{eq:Value_function_definition}
\end{equation}
Assuming the infimum is attained by an optimal control function $u^*(\cdot)\in\mathcal{U}_{\mathrm{ad}}$,  we obtain $V(x)=J_{x}(u^*(\cdot))$ \cite{vinter2010optimal}.
% Under [assumptions], the dynamic programming principle yields that the value function is the unique viscosity solution of the stationary Hamilton-Jacobi-Bellman equation
% \begin{equation}
%     \inf_{u} \left\{ \mathcal{L}(x,u) + \nabla V(x)^{\top} f(x,u) \right\} = 0.
%     \label{eq:HJB_equation}
% \end{equation}
The dynamic programming principle relates the value function to the stationary Hamilton--Jacobi--Bellman (HJB) equation
\begin{equation}
    \inf_{u\in U}
    \left\{
        \mathcal{L}(x,u)
        +\nabla V(x)^{\top}f(x,u)
    \right\}=0.
    \label{eq:HJB_equation}
\end{equation}
When $V(x)$ is continuously differentiable, this equation is understood in the classical sense and $\nabla V$ is the gradient of the optimal value function with respect to the state. In general, the value
function need not be differentiable, and the HJB equation must instead be interpreted in the viscosity sense
\cite{vinter2010optimal}.
If $V\in C^1(\mathbb{R}^n)$ is known, then an optimal input at any state $x \in \mathbb{R}^n$ is given by a feedback law \cite{Liberzon2012}
\begin{equation}
    u^*(x)=\kappa(x) \in \arg\min_{u\in U} \left[\mathcal{L}(x,u) + \nabla V(x)^T f(x,u) \right].
    \label{eq:optimal_control_argument}
\end{equation}

\begin{remark}
    We consider the infinite-horizon setup because it yields a value that depends only on the state.
    This is the simplest setting for BASOC, though generalization to finite-horizon problems with time-dependent value functions is possible.
\end{remark}

The central challenge is the extreme difficulty of finding value functions on high-dimensional state spaces.
Discretized systems in fluid dynamics that we would like to control routinely have state vectors with $10^6$ dimensions or more.
The value functions on such high-dimensional state spaces become extremely difficult to approximate, compute, and store owing to the well-known ``curse of dimensionality'' and associated challenges with solving the HJB equation in high dimensions.

\subsection {Pontryagin Minimum Principle} The Pontryagin Minimum Principle \cite{Pontryagin1962} derived from the calculus of variation provides first-order necessary optimality conditions for the optimal control problem \eqref{eq:Lagrange_cost_functional}. Following the standard convention, we define a Hamiltonian $\mathcal{H} : \mathbb{R}^n  \times \mathbb{R}^m \times \mathbb{R}^n \to \mathbb{R}$.
\begin{equation}
    \mathcal{H}(x(t), u(t), \lambda(t)) := \mathcal{L}(x(t), u(t)) + \lambda(t)^{\top} f(x(t), u(t)), \quad t \in [0,\infty)
    \label{eq:Hamiltonian}
\end{equation}
where $\lambda(t) \in \mathbb{R}^n$ is a Lagrange multiplier (costate) variable enforcing the system dynamics. If $(u^*(t), x^*(t))$ is an optimal trajectory then there exists a costate trajectory $\lambda^*(t)$ so that the following canonic equations are satisfied:
\begin{align}
\dot{x}^*(t) &= \mathcal{H}_{\lambda}(x^*(t), u^*(t), \lambda^*(t)),
\quad x^*(0) = x_{0},
\label{eq:forward_dynamics_state_equation}
\\[0.5em]
\dot{\lambda}^*(t) &= -\mathcal{H}_{x}(x^*(t), u^*(t), \lambda^*(t)),
\quad \lim_{t \to \infty} \lambda^*(t) = 0,
\label{eq:backward_dynamics_adjoint_equation}
\end{align}
with the stationarity condition 
\begin{align}
   \mathcal{H}_{u}\bigl(x^{*}(t), u^{*}(t), \lambda^{*}(t)\bigr) = 0, \quad \forall t \in [0,\infty),
\\[0.5em]
   \mathcal{H}\bigl(x^{*}(t), u^{*}(t), \lambda^{*}(t)\bigr) \leq  \mathcal{H}\bigl(x^{*}(t), u(t), \lambda^{*}(t) \bigr), \quad \forall t \in [0,\infty).
\end{align}
Solving these necessary conditions for $(u^*(t), x^*(t), \lambda^*(t))$ is a two-point boundary value problem where $x^*(0)$ in \eqref{eq:forward_dynamics_state_equation} is the initial condition for the state ODE at optimality and $\lim_{t \to \infty} \lambda^*(t) = 0$ satisfies the transversality condition for the adjoint ODE computed backwards in time. The transversality condition follows from the absence of a terminal cost and the free terminal state at infinity \cite{CannarsaFrankowska2017}. 
The key fact for BASOC is that at points where the value function is differentiable, the costate (adjoint) variable and the gradient of the optimal value function are related by 
\begin{equation}
    \lambda^*(t) = \nabla V(x^*(t)).
    \label{eq:adjoint_value_function_equality}
\end{equation}
Under standard regularity assumptions, the value function is locally Lipschitz continuous, and Rademacher's theorem ensures that it is differentiable almost everywhere \cite{Bardi1997,vinter2010optimal}. At points where the value function is not differentiable, the relationship between the costate and the value function is understood in terms of generalized gradients \cite{CannarsaFrankowska2017,bensoussan2023valuegradient,ClarkeVinter1987}. The relation between the adjoint variable and the gradient of the value function is the link between the HJB equation and the Pontryagin minimum principle. This connection will be exploited in computing the balanced active subspace.

\section{Balanced Active Subspaces for Optimal Value Functions}
\label{sec:Algorithm_for_Projection}
We compute balanced active variables and projections for the value function \eqref{eq:Value_function_definition} along optimal trajectories with initial conditions drawn from a probability measure $\mu$. For example this could be the Sinai–Ruelle–Bowen (SRB) \cite{Young2002SRB} measure on an attractor representing uncontrolled states such as separated or turbulent flow, from which the objective is to optimally drive the system towards attached or laminar flow. To do this, we sample initial conditions from the distribution $\mu$ and compute optimal trajectories from these using a standard forward-backward sweep method (FBSM) detailed in Algorithm~\ref{alg:Trajectory_optimization}. The optimal state and costate samples from this ensemble are used within Algorithm~\ref{alg:BASOC_Projection_Algorithm} to compute biorthogonal direct and adjoint modes defining the BASOC projection and reduced variables. In practice, trajectory optimization may not yield a globally optimal solution. In this case the locally optimal state and costate trajectories used by BASOC correspond to state-constrained optimal value functions.

Specifically, we seek a rank-$r$ linear projection $P=\Phi_r\Psi_r^T \in \R^{n\times n}$ defined by direct modes $\Phi_r \in \R^{n\times r}$ and adjoint modes $\Psi_r^T$ satisfying the biorthogonality condition $\Psi_r^T \Phi_r = I_r$. This projection seeks to jointly preserve directions in state space along which the value function $V(x)$ is most sensitive and states exhibit large variance. 
We quantify the distribution of states $x(t)^*$ along optimal trajectories of \eqref{eq:Continous_time_dynamics} with the state covariance matrix,
\begin{equation}
W_x = \E_{x_0\sim \mu} \!\left[ \int_{0}^{\infty} x^{*}(t)\,x^{*}(t)^{T}\,dt \right].
\label{eq:State_covariance_matrix}
\end{equation}
The sensitivity of the optimal value function is quantified by its gradient covariance matrix along optimal trajectories, which by \eqref{eq:adjoint_value_function_equality} is the optimal costate covariance,
\begin{equation}
W_{g} = \E_{x_0\sim \mu} \!\left[\int_{0}^{\infty} \nabla_{x}V(x^*(t))\,\nabla_{x}V(x^*(t))^{T}\,dt \right]
= \E_{x_0\sim \mu} \!\left[\int_{0}^{\infty} \lambda^{*}(t)\,\lambda^{*}(t)^{T}\,dt \right].
\label{eq:Gradient_covariance_matrix}
\end{equation}
Under an invertible linear change of coordinates on state space $x(t) = T \tilde{x}(t)$ these matrices transform into $\tilde{W}_x = T^{-1} W_x T^{-T}$ and $\tilde{W}_g = T^T W_g T$.
By a well-known result stated in \cite{Zhou1996robust} Theorem~3.22, there is a balancing transformation $T$ where these covariances are simultaneously diagonalized (balanced) as
\begin{equation}
    \tilde{W}_x = \diag(\Sigma, \Sigma_1, 0, 0), \qquad \tilde{W}_g = \diag(\Sigma, 0, \Sigma_2, 0),
\end{equation}
where $\Sigma, \Sigma_1, \Sigma_2$ are diagonal positive-definite matrices with the balanced singular values in $\Sigma = \diag(\sigma_1, \ldots, \sigma_{n_b})$ arranged in non-increasing order.
The balancing projection $P = T \diag(I_r, 0) T^{-1}$
corresponds to transforming into the balanced coordinate system, truncating to retain the leading $r \leq \rank \Sigma$ balanced coordinates, and then mapping back to the original coordinate system.
Square-root balancing \cite{tombs1987truncated, antoulas2005approximation} computes the direct and adjoint modes of this projection and the balanced singular values from factorized covariance matrices $W_x = X X^T$ and $W_g = Y Y^T$ via slim singular value decomposition (SVD) and rank-$r$ truncation:
\begin{equation}\label{eq:sqrt_balancing}
    Y^T X = U \Sigma V^T, \qquad \Phi_r = X V_r \Sigma_r^{-1/2}, \qquad \Psi_r = Y U_r \Sigma_r^{-1/2}.
\end{equation}

\begin{algorithm}[H]
\caption{Balanced Active Subspace (BASOC) Projection Algorithm}
\label{alg:BASOC_Projection_Algorithm}
\textbf{Input:} $\mu$ (sampling measure), $N_{\mathrm{IC}}$ (number of samples), $r$ (reduced dimension)\\[0.3em]
\textbf{Output:} $\Phi_r$ (primal modes), $\Psi_r$ (adjoint modes), $\Sigma_r$ (dominant singular values)
\begin{algorithmic}[1]
\STATE Draw $x_0^{(i)} \sim \mu$ for $i=1,\ldots,N_{\mathrm{IC}}$.
\FOR{$i = 1,\ldots,N_{\mathrm{IC}}$}
    \STATE Solve the optimal control problem from $x_0^{(i)}$ using Algorithm~\ref{alg:Trajectory_optimization}.
    \STATE Obtain the optimal state and costate snapshot matrices $\widetilde{X}^{(i)}$ and $\widetilde{Y}^{(i)}$ as defined in \eqref{eq:Single_Gradient_Trajectory} and \eqref{eq:Single_State_Trajectory}.
\ENDFOR
\STATE Form the aggregate snapshot matrices
\vspace{-0.8em}
\[
X=\frac{1}{\sqrt{N_{\mathrm{IC}}}}[\widetilde{X}^{(1)} \cdots \widetilde{X}^{(N_{\mathrm{IC}})}] \qquad
Y=\frac{1}{\sqrt{N_{\mathrm{IC}}}}[\widetilde{Y}^{(1)} \cdots \widetilde{Y}^{(N_{\mathrm{IC}})}].
\]
\vspace{-1em}
\STATE Compute a slim SVD of the state--costate inner-product matrix
\[
    Y^{T}X = U \Sigma V^{T}
\]
\vspace{-1.8em}
\STATE Construct the biorthogonal primal and adjoint modes
\[
\Phi = X V \Sigma^{-1/2}, \qquad
\Psi = Y U \Sigma^{-1/2}.
\]

\STATE Retain the first $r$ BASOC modes and singular values:
\vspace{-0.em}
\[
\Phi_r = \Phi(:,1:r),\qquad
\Psi_r = \Psi(:,1:r),\qquad
\Sigma_r = \Sigma(1:r,1:r).
\]

\STATE Return $\Phi_r$, $\Psi_r$, and $\Sigma_r$.
\end{algorithmic}
\end{algorithm}

\begin{remark}
If the aggregate snapshot matrices $X, Y \in \mathbb{R}^{n \times M} $ are wide ($n \ll M $), where $n$ is the state dimension and $M$ is the total number of snapshots, an optional pre-compression can be performed prior to step 7 in Algorithm~\ref{alg:BASOC_Projection_Algorithm}. Specifically, compute the slim SVDs 
\begin{equation}
    X = U_X \Sigma_X V_X^{T},
\qquad
Y = U_Y \Sigma_Y V_Y^{T},
\end{equation}
and replace 
\begin{equation}
    X \leftarrow U_X \Sigma_X, \qquad Y \leftarrow U_Y \Sigma_Y.
\end{equation}
Since the right singular factors $V_X$ and $V_Y$ do not affect the column space of $X$ and $Y$, they may be removed. Additionally, a truncated rank-$\ell$ approximation may be employed: 
\begin{equation}
    X \approx U_{X,\ell}\Sigma_{X,\ell},
    \qquad
    Y \approx U_{Y,\ell}\Sigma_{Y,\ell},
\end{equation}
where $\ell \ll n$, while retaining the dominant variance and sensitivity directions. 
\end{remark}
% We refer to Zahm et al \cite{Zahm2020} for the generalization of selecting this projection to multivariate vector-valued functions.

Given an arbitrary linear projection $P$, the optimal approximation of the value function $V(x)$ in the mean-square sense is given by its conditional expectation,
\begin{equation}
    \tilde{V}(Px) = \E \bigl [V(x)\mid Px \bigr].
    \label{eq:Value_func_approximation}
\end{equation} 
This averages $V(x)$ over the components of $x$ in the null space of the projection $P$. 
Zahm et al. \cite{Zahm2020} show that when state distribution is Gaussian, the resulting error is bounded by
\begin{equation}
\E\left[ \left\| V(x) - \tilde{V}(Px) \right\|^2 \right] 
\leq \Tr\left[ W_g (I - P)\, W_x (I - P)^{T} \right].
\label{eq:Value_func_error_bound}
\end{equation}
Otto et al. \cite{Otto2023model} show that the balancing projection minimizes this upper bound.
The truncated balanced active variables
\begin{equation}\label{eq:State_Projection}
    z = \Psi_r^T x
\end{equation}
are thus a principled choice for constructing reduced-order approximations of the value function $V(x) \approx \hat{V}(\Psi_r^T x) = \hat{V}(z)$.
Nonlinear kernel-based generalizations of these coordinates can also be found \cite{Otto2023model, romor2022kernel}, though we focus on the linear case here.

BASOC computes the balancing projection and balanced active variables for empirical covariance matrices of state and costate snapshots along the sampled ensemble of locally optimal trajectories.
The square-root balancing method is used in Algorithm~\ref{alg:BASOC_Projection_Algorithm} with factors $X$ and $Y$ obtained from Monte Carlo sampling of initial conditions and temporal quadrature.
With $w_j$ being quadrature weights at quadrature points $t_j$ approximating integration over $[0, \infty)$, the blocks of $X$ and $Y$ corresponding to the $i$th trajectory are formed by 
\begin{equation}
    \tilde{Y}^{(i)} =\begin{bmatrix}| & | & | \\\sqrt{w_1}\,\lambda^{*(i)}(t_1) &\sqrt{w_2}\,\lambda^{*(i)}(t_2) & \cdots \;\sqrt{w_{N_q}}\,\lambda^{*(i)}(t_{N_q}) \\
    | & | & |\end{bmatrix}
    \in \mathbb{R}^{n \times N_q},
    \label{eq:Single_Gradient_Trajectory}
\end{equation}
\begin{equation}
    \tilde{X}^{(i)} = \begin{bmatrix} | & | & | \\
    \sqrt{w_1}\,x^{*(i)}(t_1) & \sqrt{w_2}\,x^{*(i)}(t_2) & \cdots \;\sqrt{w_{N_q}}\,x^{*(i)}(t_{N_q}) \\
    | & | & | \end{bmatrix}
    \in \mathbb{R}^{n \times N_q},
    \label{eq:Single_State_Trajectory}
\end{equation}
where $N_q$ is the number of temporal quadrature nodes and $N_{\mathrm{IC}}$ is the number of sampled initial conditions. 
Specifically, we use the same linearly-stretched left Gauss-Radau quadrature rule that we use for FBSM over a suitably long time horizon, as suggested in \cite{ross2012pseudospectral}.
These yield factors of empirical covariance matrices because
\begin{equation}
W_x \approx \frac{1}{N_{\mathrm{IC}}} \sum_{i=1}^{N_{\mathrm{IC}}} \sum_{j=1}^{N_q} w_j\, x^{*(i)}(t_j) \left(x^{*(i)}(t_j)\right)^T = XX^T,
    \label{eq:emperical_state_covariance}
\end{equation}
\begin{equation}
W_g
\approx \frac{1}{N_{\mathrm{IC}}} \sum_{i=1}^{N_{\mathrm{IC}}} \sum_{j=1}^{N_q} w_j\, \lambda^{*(i)}(t_j) \left(\lambda^{*(i)}(t_j)\right)^T = YY^T.
    \label{eq:emperical_gradient_covariance}
\end{equation}

\section{Exploiting BASOC Projections and Variables}
\subsection{Petrov-Galerkin modeling}
The BASOC projections $P = \Phi_r \Psi_r^T$ in Section~\ref{sec:Algorithm_for_Projection} can be used to construct $r$-dimensional ROMs of the full-order model (FOM) \eqref{eq:Continous_time_dynamics}. One of the simplest ways to do this is by Petrov--Galerkin projection of full-order dynamics. The resulting ROM is an $r$-dimensional dynamical system given by
\begin{equation}
    \left\{
    \begin{aligned}
    \dot{z} &= \Psi_r^T f(\Phi_r z,\; u),
    \qquad z(0) = \Psi_r^T x_0, \\
    \hat{y} &= h(\Phi_r z,\; u).
    \end{aligned}
    \right.
    \label{eq:Reduced_order_model}
\end{equation}
One could also construct nonlinear Galerkin ROMs \cite{marion1989nonlinear, buchfink2024model, otto2023learning} based on nonlinear BASOC projections of the form 
\begin{equation}
    P_{\text{nonlin.}}(x) = \Phi_r \Psi_r^T x + (I - \Phi_r \Psi_r^T) \phi(\Psi_r^T x),
\end{equation}
where $\phi$ parameterizes a nonlinear trial manifold as a graph in balanced coordinates and can be fit to trajectory data by $W_g$-weighted least-squares or losses as in \cite{otto2023learning}. The BASOC-Petrov--Galerkin method is intrusive as it requires directly accessing the FOM governing equations \eqref{eq:Continous_time_dynamics}. However, the BASOC projection itself can be computed using less intrusive automatic differentiation methods for forward-backward sweeping. Thus, operator inference \cite{peherstorfer2016data, aretz2026nested, Geelen2023} using the BASOC projection or a nonlinear extension becomes an attractive option for non-intrusive model reduction. We leave this and the nonlinear extension for future work.

To provide some theoretical motivation for BASOC-Petrov--Galerkin ROMs, we consider the simplified case of a linear FOM and quadratic cost, that is, the LQR problem. In this setting, \cref{thm:BASOC_for_LQR} shows that BASOC is related to initial condition balancing \cite{Beattie2017} of the optimal closed-loop system.
Specifically, let the FOM be
    \begin{equation}
    \dot{x}(t) = Ax(t) + Bu(t),
    \qquad
    x(0)=x_0,
    \label{eq:Linear_dynamical_system}
\end{equation}
where $A \in \mathbb{R}^{n\times n}$, $B \in \mathbb{R}^{n\times m}$, together with the LQR cost functional,
\begin{equation}
J(u(\cdot)) := \int_0^\infty \left( x(t)^T Q x(t) + u(t)^T R u(t) \right) \, dt,
\label{eq:Quadratic_cost_function}
\end{equation}
where $Q \succeq 0$ defines the state penalty and ${R} \succ 0$ defines the control penalty. 
When $(A,B)$ is stabilizable and $(Q^{1/2},A)$ is detectable, the optimal control law and value function \cite{Zhou1996robust, Liberzon2012} are given by
\begin{equation}\label{eq:linear_optimal_control_and_value_func}
    u^*(t) = -K\,x(t), \qquad V(x)=x^{T} \Pi x,
\end{equation}
where $K = R^{-1} B^T \Pi$ is the optimal state-feedback gain matrix and $\Pi \succeq 0$ is the unique stabilizing solution of the algebraic Riccati equation
\begin{equation}
    A^T\Pi + \Pi A - \Pi B R^{-1} B^T \Pi + Q = 0.
    \label{eq:riccati_equation}
\end{equation}
We use the following well-known lemma. Its proof is hard to locate, so we provide it.
\begin{lemma}[uniqueness of square-root balancing]\label{lem:sqrt_balancing_uniqueness}
    Let $W_c, W_o \in \R^{n\times n}$ be positive semidefinite matrices, let $W_c = X X^T$ and $W_o = Y Y^T$ be any factorizations of them, and let $Y^T X = U \Sigma V^T$ be a slim SVD with nonzero singular values $\sigma_1 \geq \cdots \sigma_\rho > 0$. Then the eigenvalues of $W_c W_o$ are $\lambda_j = \sigma_j^2$ for $j \leq \rho$ and $\lambda_j = 0$ for $j > \rho$, meaning $\Sigma$ is uniquely determined by $W_c W_o$. Suppose that $\sigma_r > \sigma_{r+1}$ and let $\Phi_r = X V_r \Sigma_r^{-1/2}$ and $\Psi_r = Y U_r \Sigma_r^{-1/2}$ be defined by rank-$r$ truncation of the slim SVD. Then the rank-$r$ square-root balancing projection $P_r = \Phi_r \Psi_r^T$ is the unique rank-$r$ spectral projector of $W_c W_o$ associated with its largest $r$ eigenvalues.
\end{lemma}
\begin{proof}
    Letting $\Phi = X V \Sigma^{-1/2}$ and $\Psi = Y U \Sigma^{-1/2}$, we observe that
    \begin{equation}
        \Psi^T \Phi = \Sigma^{-1/2} U^T Y^T X V \Sigma^{-1/2} = I.
    \end{equation}
    Moreover, we observe that
    \begin{equation}
        (W_c W_o) 
        = X X^T Y Y^T 
        = X V \Sigma U^T Y^T
        = \Phi \Sigma^2 \Psi^T.
    \end{equation}
    Using these two relations, we obtain
    \begin{equation}
        (W_c W_o) \Phi = \Phi \Sigma^2 \qquad \mbox{and} \qquad (W_c W_o)^T \Psi = \Psi \Sigma^2,
    \end{equation}
    meaning that $(W_c W_o)$ has eigenvalues $\lambda_1 = \sigma_1^2 , \ldots, \lambda_\rho = \sigma_\rho^2$ with $\lambda_j = 0$ for all $j > \rho$.
    Because $\lambda_r > \lambda_{r+1}$, the right eigenspace $E_R = \Range(\Phi_r)$ of $W_c W_o$ is unique and so is the left eigenspace $E_L = \Range(\Psi_r)$.
    Because a projection is uniquely determined by its range and nullspace (see Section~5.9 \cite{Meyer2000matrix}) it follows that $P_r = \Phi_r \Psi_r^T$ is uniquely determined by $W_c W_o$ and corresponds to its leading rank-$r$ spectral projector.
\end{proof}

\begin{theorem}\label{thm:BASOC_for_LQR}
    Suppose that $(A,B)$ is stabilizable, $(Q^{1/2},A)$ is detectable, and let $K$ and $V$ be as in \eqref{eq:linear_optimal_control_and_value_func}.
    Let the initial conditions $x_0 \sim \mu$ have finite $W_{IC} = \E x_0 x_0^T$ and let $P_{\text{BASOC}}$ be a rank-$r$ BASOC projection defined by the state and costate covariances \eqref{eq:State_covariance_matrix} and \eqref{eq:Gradient_covariance_matrix}.
    Let $P_{\text{BT}}$ be a rank-$r$ balancing projection for the realization
    \begin{equation}
        \mathcal{G}_{\mathrm{cl}} =
        \left[
        \begin{array}{c|c} 
            A-BK & W_{IC}^{1/2}\vphantom{\Big|} \\ 
            \hline 
            Q^{1/2} & 0\vphantom{\Big|} \\  
            R^{1/2} K & 0
        \end{array}
        \right].
    \label{eq:Linear_system_realization}
    \end{equation}
    The Hankel singular values $\tilde{\Sigma}$ of $\mathcal{G}_{\mathrm{cl}}$ and the BASOC singular values $\Sigma$ are related by
    \begin{equation}
        \Sigma = 2 \tilde{\Sigma}^2.
    \end{equation}
    If these have a gap $\sigma_r > \sigma_{r+1}$, then the two projections are the same

    \begin{equation}\label{eq:Projection_equivalence}
    P_{BASOC} = P_{BT}.
    \end{equation} 
\end{theorem}
\begin{proof}
The algebraic Riccati equation
\eqref{eq:riccati_equation} can be rewritten as the Lyapunov equation
\begin{equation}
    (A-BK)^T\Pi+\Pi(A-BK)+Q+K^TRK=0.
    \label{eq:Lyapunov_riccati_equation_closed_loop}
\end{equation}
The observability Gramian $W_{o}$ for the realization in \eqref{eq:Linear_system_realization} satisfies the same Lyapunov equation
\begin{equation}
(A-BK)^{T}W_o + W_o(A-BK) + Q +K^{T}RK =0.
\label{eq:Observability_gramian_Lyapunov}
\end{equation}
Because $A - B K$ is Hurwitz, the solution is unique and we must have
\begin{equation}
    \Pi = W_{o}.
    \label{eq:Riccati_equal_observability_gramian}
\end{equation}
With initial condition $x_0$, the optimal state trajectory solving the LQR problem is
\begin{equation}
    x^*(t)=e^{(A-BK)t}x_0.
    \label{eq:Closed_loop_system}
\end{equation}
Hence, the state covariance matrix along optimal trajectories in \eqref{eq:State_covariance_matrix} is
\begin{equation}
    W_x = \E_{x_0} \left[ \int_0^\infty x^*(t)x^*(t)^T \, dt \right]
    =
    \E_{x_0} \left[ \int_0^\infty e^{(A-BK)t} x_0 x_0^T e^{(A-BK)^T t} \, dt\right].
    \label{eq:Wx_optimal_trajectories_linear}
\end{equation}
By Fubini's theorem and linearity of expectation, the state covariance matrix is equal to the controllability Gramian $W_c$ of the realization in \eqref{eq:Linear_system_realization},
\begin{equation}
    W_x = \int_0^\infty e^{(A-BK)t} \, W_{IC} \, e^{(A-BK)^T t}\, dt = W_c.
    \label{eq:Wx_linearity_of_expectation}
\end{equation}
Because the costates along optimal trajectories are $\lambda^*(t) = \grad V(x^*(t)) = 2\Pi x^*(t)$, 
the costate covariance matrix in \eqref{eq:Gradient_covariance_matrix} is  
\begin{equation}
    W_g = \E_{x_0} \left[ \int_0^\infty \lambda^*(t)\lambda^*(t)^T \, dt \right]
    = \E_{x_0} \left[ \int_0^\infty 4\Pi x^*(t)x^*(t)^T \Pi \, dt \right] = 4W_o W_x W_o.
    \label{eq:Gradient_covariance_linear_system}
\end{equation}
Therefore, we have
\begin{equation}
    W_x W_g = 4 (W_c W_o)^2,
\end{equation}
and the claimed results follow from \cref{lem:sqrt_balancing_uniqueness}.
\end{proof}

The oblique BASOC projection modes obtained in the linear case can be interpreted as an initial-condition balancing procedure, similar to \cite{Beattie2017}. The resulting ROM approximates the system output and optimal control signals using weights derived from the cost functional along optimal trajectories.

\subsection{Surrogates for suboptimal control}
\label{subsec:Surrogate_control}
Response surfaces and surrogate models can be constructed for the value function and feedback law in the reduced variables $z = \Psi_r^T x \in \mathbb{R}^r$ rather than in the full state $x \in \mathbb{R}^n$. If these quantities admit accurate approximations, then this can dramatically reduce the parametric complexity of the resulting surrogate models, the computational cost to evaluate them, and the sample complexity to fit them.
By their construction, we expect BASOC variables to be well suited to constructing surrogate models of the optimal value function.
This can be done either by fitting a response surface to the data along optimal trajectories $\left\{ \left( \Psi_r^T x^*_i,\; V(x^*_i) \right) \right\}_{i=1}^{d}$, where $d$ denotes the total number of training samples, or by solving the HJB equation using a reduced-order model of the dynamics in the reduced BASOC coordinates. We adopt the former here.

The control law can also be approximated in reduced BASOC coordinates. For control-linear systems of the form \(\dot{x}=f(x)+Bu\) with a running cost \(\mathcal{L}(x,u)=\ell(x)+u^TRu\), where \(R\succ0\), a differentiable value function, and no control constraints, the optimal feedback law obtained from the HJB equation is \(u^*(x)=-\frac12R^{-1}B^T\nabla V(x)\). Consequently, if the value function admits an accurate approximation \(V(x)\approx\tilde V(\Psi_r^Tx)\) with sufficiently accurate gradients, the optimal feedback law can also be approximated using only the reduced coordinates \(z=\Psi_r^Tx\), since \(\nabla V(x)\approx\Psi_r\nabla_z\tilde V(z)\).
One could potentially use second derivatives of the value function obtained by solving differential Riccati equations along optimal trajectories to find balanced active variables directly for the optimal feedback law, though we leave this for future work.

A surrogate control law can be obtained directly by fitting the control input as a function of the reduced variables $z=\Psi_r^T x$. We adopt this approach by fitting a surrogate model $u = \tilde{\kappa}(z)$ to the optimal trajectory data $\left\{ \left( \Psi_r^T x^*_i,\; u^*(x^*_i) \right) \right\}_{i=1}^{d}$. While a variety of regression techniques could be used, we employ Gaussian-process regression with Matérn kernels. Further implementation details are provided in the numerical results section where we evaluate surrogate feedback laws.
An alternative approach could be to use a value function surrogate $\tilde{V}(z)$ in BASOC coordinates and Petrov-Galerkin ROM to solve the following finite-dimensional optimization problem at every time step:
\begin{equation}
    \hat{u}^*(z) \in \arg\min_{w} \left\{ \nabla \tilde{V}(z)^{T}{f}_{r}(z,w) + \mathcal{L}(\Phi_r z,w) \right\}.
    \label{eq:surrogate_control_law}
\end{equation}
However, this can be costly to perform online so we do not adopt this approach.

\section{Numerical Examples and Evaluation}
In the following numerical examples, we consider the infinite-horizon
optimal control problem
\begin{equation}
\min_{u(\cdot)} \quad J_{x_0}(u(\cdot)) =
\int_0^\infty
\left(y(t)^Ty(t)+ 100\, u(t)^T u(t)\right)\,dt,
\label{eq:numerical_optimal_control_problem}
\end{equation}
subject to the problem-specific dynamics and output quantities $y=Cx$.
\subsection{Challenging nonlinear toy problem}
\label{subsec:toy_problem_compute_projections}
We consider the challenging nonlinear model in \cite{Otto2023model} for which standard model reduction methods perform poorly:  
\begin{equation}\label{eq:Toy_model_state_space}
\begin{aligned}
\dot{x}_1 &= -x_1 + 20x_1x_3 + u, \\
\dot{x}_2 &= -2x_2 + 20x_2x_3 + u, \\
\dot{x}_3 &= -5x_3 + u, \\
y &= x_1 + x_2 + x_3.
\end{aligned}
\end{equation}
Model reduction and value function approximation for this system are challenging because the rapidly-decaying state $x_3$ strongly influences the dynamics of $x_1$ and $x_2$ via quadratic nonlinear interactions, which lead to transient energy growth for these state variables. As a result, conventional methods such as balanced truncation perform poorly since it only considers the linear part of \eqref{eq:Toy_model_state_space}. Additionally, $x_3$ remains small and has low energy due to its rapid decay and is consequently ignored by POD, which retains the most energetic variables closely aligned with $x_1$ and $x_2$. The control goal is to minimize the cost functional in \eqref{eq:numerical_optimal_control_problem}, suppress transient growth, and drive the system to the origin faster than its natural dynamics via feedback control.

To compute the balanced active modes defining the projection $P$, we performed trajectory optimization in Algorithm~\ref{alg:Trajectory_optimization} for 50 initial conditions sampled uniformly at random from the cube $[0.5,1]^3\subset\R^3$, with $N_q=200$ quadrature points over the time interval $0 \leq t \leq 10$. The optimal state and gradient snapshots obtained via trajectory optimization were used to assemble the state and gradient covariance matrices. The projection was computed via Algorithm~\ref{alg:BASOC_Projection_Algorithm}. The rapidly decaying balanced singular values indicate that the value function can be well-approximated using two balanced active variables. 
With $r=2$, the BASOC singular value tail satisfies
\begin{equation}
\frac{\sum_{\ell=1}^{r}\sigma_{\ell}^{2}}{\sum_{\ell=1}^{n}\sigma_{\ell}^{2}} \geq 0.99.
\label{eq:cumulative_balanced_singular_value_ratio}
\end{equation}
We compare BASOC against other methods of obtaining rank-$2$ linear projections $P = \Phi \Psi^T$ and reduced variables $z = \Psi^T x$ to approximate the value function, optimal control law, and construct Petrov-Galerkin ROM. To evaluate these projections on control-relevant tasks, we test on 30 unseen initial conditions sampled uniformly at random from the same cube $[0.5,1]^3$, and 30 extrapolated initial conditions sampled uniformly at random from the new cube $[1,1.5]^3$.
\begin{figure}[t] 
    \centering
    \includegraphics[width=0.95\linewidth]{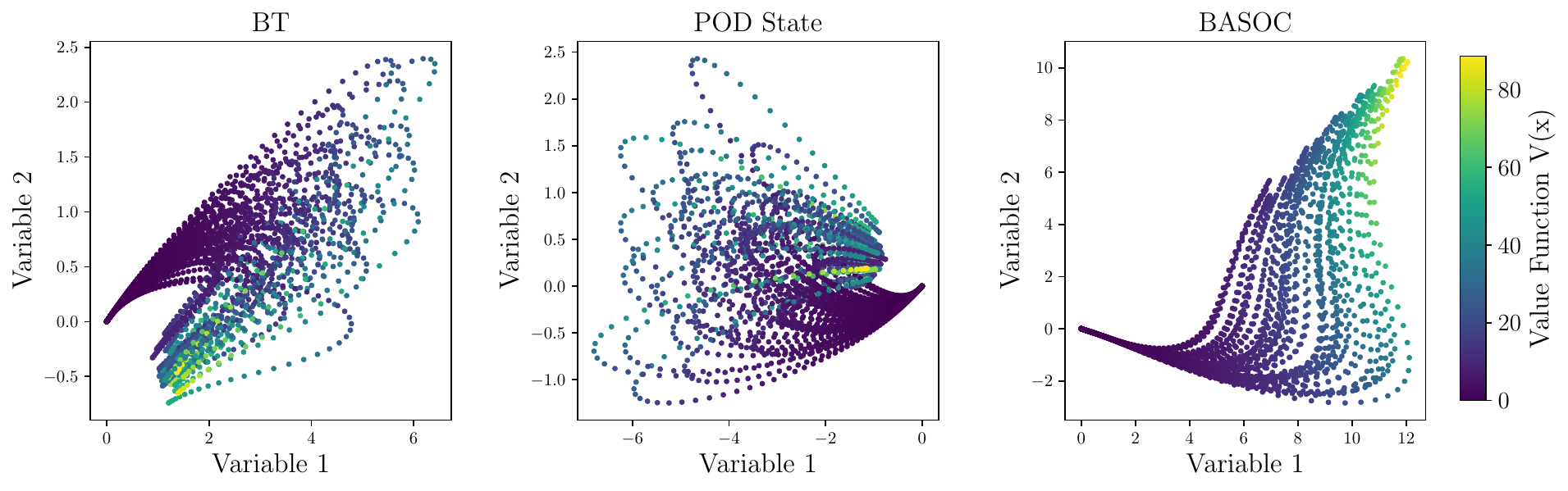}
    \caption{Comparison of the value function along optimal trajectories plotted in reduced variables selected by different methods for the challenging nonlinear toy model.}
    \label{fig:value_function_projection_toy}
\end{figure}
In Figure~\ref{fig:value_function_projection_toy}, we observe that the value function is nearly a single-valued function of the retained BASOC coordinates over the trajectory ensemble. In contrast, the value function takes distinct values at nearby points in the coordinates retained by POD and BT, meaning the value function  is poorly approximated as a function of these coordinates. The POD Adjoint coordinates --- obtained by truncated SVD of the optimal costate matrix $Y$ --- also provide an excellent approximation of the value function comparable to BASOC.

\subsection{Complex Ginzburg-Landau Equation}
\label{subsec:CGLE_compute_projections}
As a representative model for bluff-body wake flows, we consider the nonlinear complex Ginzburg-Landau (CGL) equation following \cite{Ilak2010}:
\begin{equation}
\frac{\partial q}{\partial t} = \left(\nu \frac{\partial}{\partial x}
+
\gamma \frac{\partial^2}{\partial x^2}
+ \mu \right) q - a |q|^2 q, \qquad x \in (-\infty,\infty),
\qquad q(x,t)\in\mathbb{C},
\label{eq:CGLE_Equation}
\end{equation}
with parameters $a=0.1$, $\nu = 2 + 0.4i$, $\gamma = 1 - i$, $\mu = (\mu_0 - 0.2^2) + \frac{\mu_2 x^2}{2}$, $i=\sqrt{-1}$, $\mu_0 = 0.41$, and $\mu_2 = -0.01$. For these parameters, the origin is unstable with nonnormal linearized dynamics resulting in high sensitivity in an ``upstream'' region and large response in a ``downstream'' region. Such behavior is typical of boundary-layer, shear, and convectively unstable flows \cite{Schmid2001}, making it a standard test case. In addition, the system undergoes a supercritical Hopf bifurcation with respect to the parameter $\mu_0$, giving rise to a stable limit cycle at the chosen value of $\mu_0$ \cite{Chomaz1988}. The optimal control goal is to minimize \eqref{eq:numerical_optimal_control_problem} by applying feedback forcing upstream at ``branch I'' to drive initial conditions on the stable limit cycle to the origin and suppress downstream disturbance amplification, as illustrated in Figure~\ref{fig:Controlled_and_uncontrolled_CGLE_equation}(b). The complex-valued control input \(u\) acts through an upstream Gaussian actuator kernel located upstream as: 
\begin{equation}
    b(x)=\exp\!\left[-\left(\frac{x-{x_1}}{s}\right)^2\right].
\end{equation}
 The measurement (output) matrix $C$ for output $y$ is a Gaussian kernel function: 
\begin{equation}
    C(x)=\exp\!\left[-\left(\frac{x+x_2}{s}\right)^2\right] \qquad    y(q)=\int_{-\infty}^{\infty} C(x)\,q(x)\,dx,
\end{equation}
where $s=1.6$ and $x_1 , x_2 = -\sqrt{\frac{-2(\mu_0 - 0.2^2)}{\mu_2}}$. Note that $x_1$ is the location of ``branch I'' upstream and $-x_2$ is the location of ``branch II'' downstream of the unstable region. See \cite{Ilak2010, Bagheri2009} for more details. Following \cite{Ilak2010, Bagheri2009}, a Gauss--Hermite pseudospectral discretization of \eqref{eq:CGLE_Equation}, treated as a coupled system of real PDEs due to the non-holomorphic nonlinearity, yields a finite-dimensional full-order model of the form,
\begin{align}
\dot{x} &= Ax + f_N(x) + Bu \\
y &= Cx.
\end{align}
The state $x(t) \in \mathbb{R}^{440}$ is constructed to have Euclidean norm equal to the $L^2(\R)^2$-norm of the function it represents; $f_N$ is a cubic nonlinearity; and $u(t),y(t) \in \mathbb{R}^{2}$ denote the input and output vectors, respectively, each containing the real and imaginary components of the corresponding quantities.
\begin{figure}[t]
    \centering
    \subfloat[Uncontrolled CGL model]{%
        \includegraphics[width=0.49\linewidth]{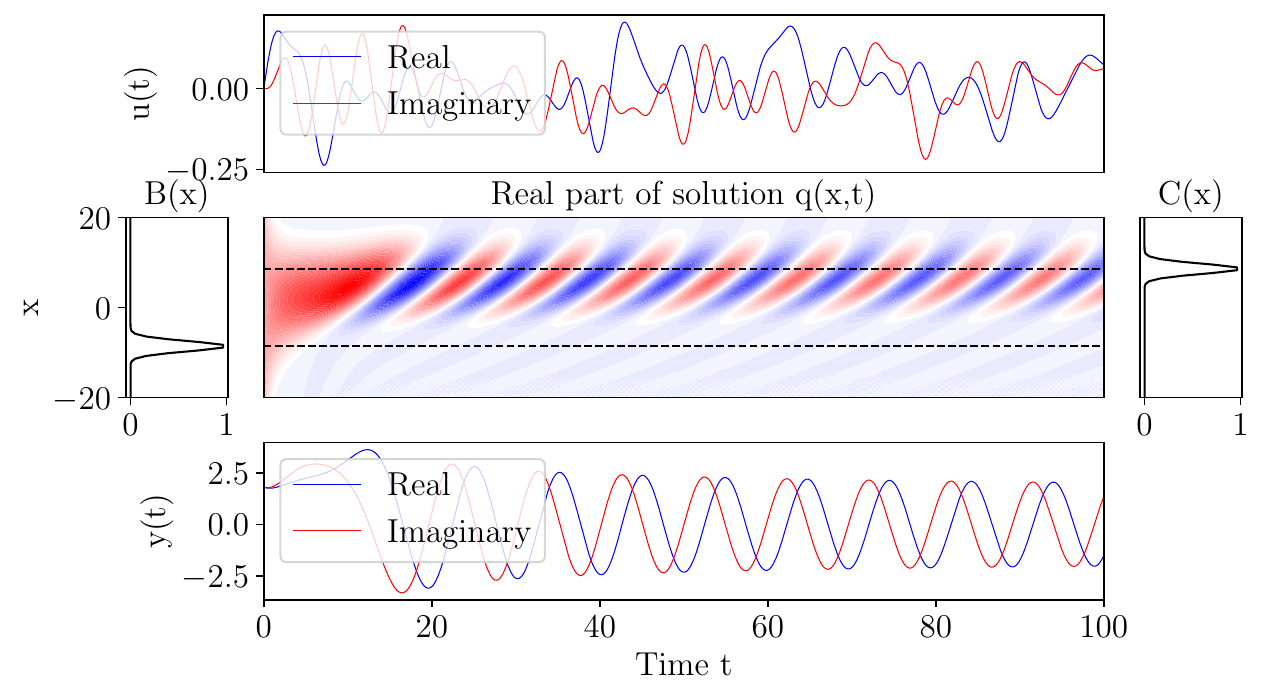}
    }
    \hfill
    \subfloat[Controlled CGL model]{%
        \includegraphics[width=0.49\linewidth]{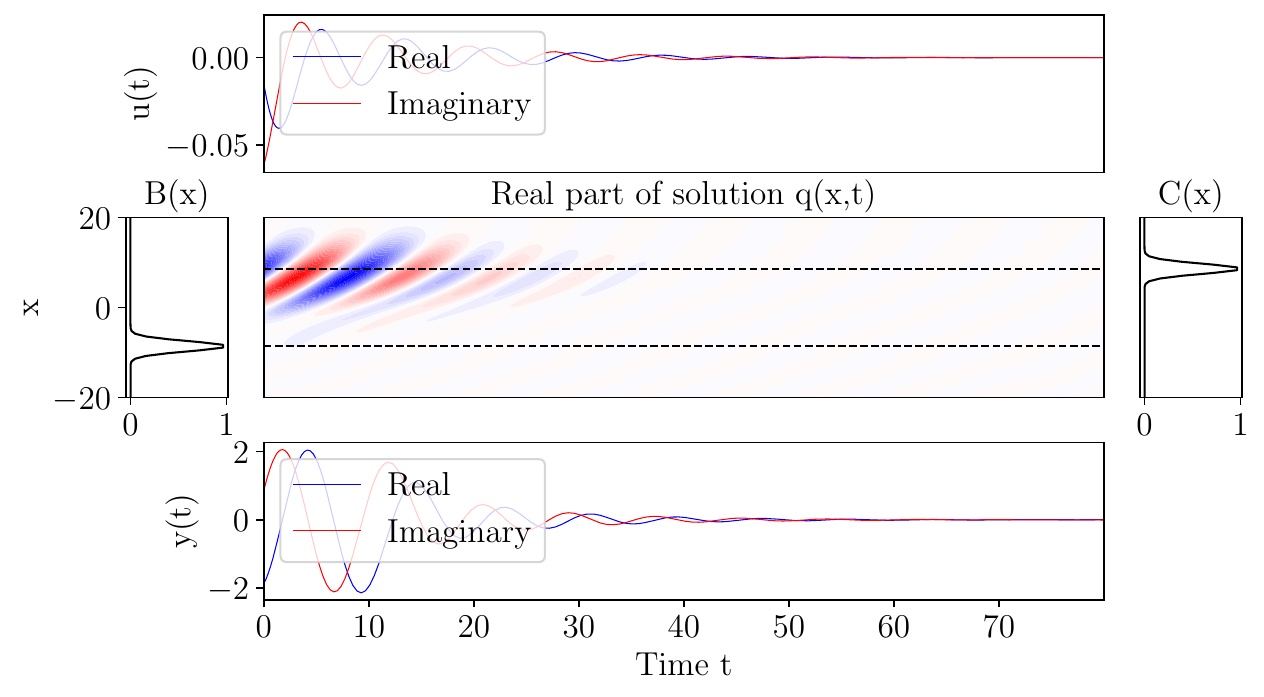}
    }
    \caption{In (a) we show the uncontrolled CGL equation exhibiting a large downstream response due to upstream disturbances, and in (b) we show the controlled CGL equation under optimal control forcing upstream.}
    \label{fig:Controlled_and_uncontrolled_CGLE_equation}
\end{figure}
To construct BASOC, POD, and POD Adjoint modes, we perform trajectory optimization over $t\in[0,80]$ from $50$ random initial conditions. These initial conditions are generated by placing independent zero-mean Gaussian coefficients with variance $(0.15)^2$ on the first $100$ $L^2(\R)$-orthonormal Gauss--Hermite functions for each component --- real and imaginary --- of the solution.
\begin{figure}[H] 
    \centering
    \includegraphics[width=0.55\linewidth]{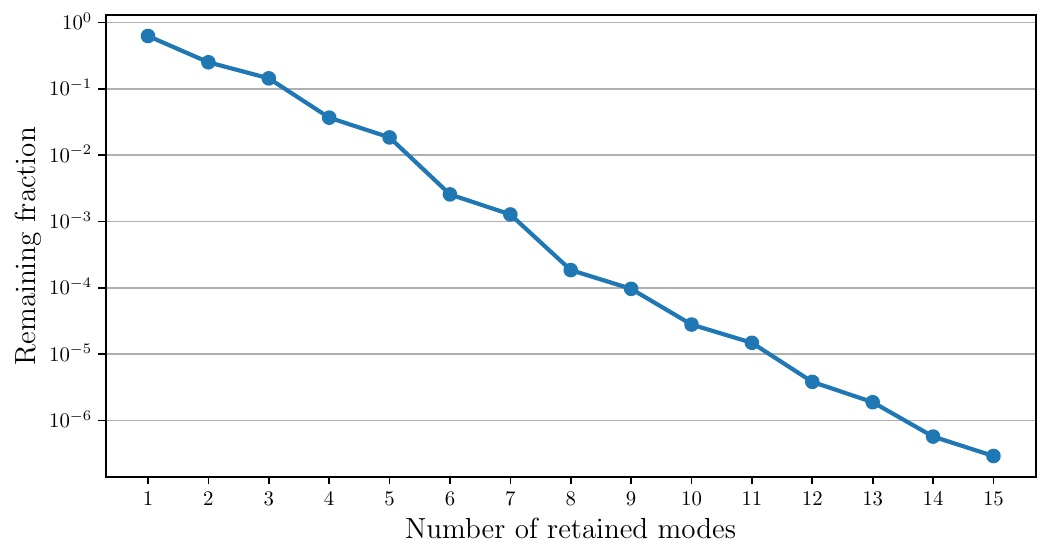}
    \caption{Remaining fraction $\sum_{\ell=r+1}^n \sigma_{\ell}^2 / \sum_{\ell=1}^{n} \sigma_{\ell}^2$ of the total squared BASOC singular values as a function of the number $r$ of retained modes for the CGL equation.}
    \label{fig:BASOC_singular_values_CGLE}
\end{figure}

Figure~\ref{fig:BASOC_singular_values_CGLE} shows that the BASOC singular values exhibit rapid decay, suggesting that the value function can be well approximated using a small number of balanced active variables.
We visualize the value function over the ensemble of training trajectories in the leading three coordinates selected by different methods in Figure~\ref{fig:value_function_projection}.
Here BT, POD Adjoint, and BASOC coordinates all separate high and low value function values, while the POD coordinates do not. Because the equilibrium at $q=0$ in the CGL equation is unstable at the chosen value of $\mu_0$, we used the balancing method for unstable systems in \cite{zhou1999balanced} to compute the BT projections. To evaluate these projections on control-relevant tasks, we test on 10 unseen random initial conditions. 
\begin{figure}[t]
    \centering
    \subfloat[3D POD State subspace]{%
        \includegraphics[width=0.49\linewidth]{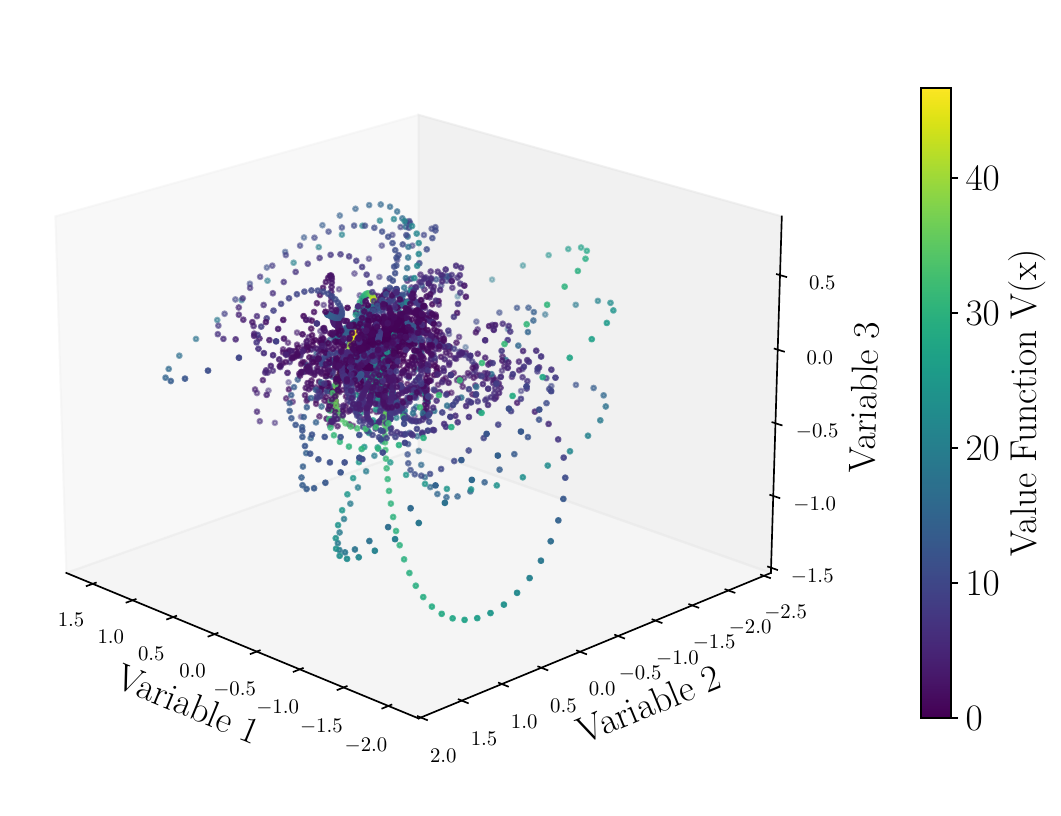}
    }
    \hfill
    \subfloat[3D BT subspace]{%
        \includegraphics[width=0.49\linewidth]{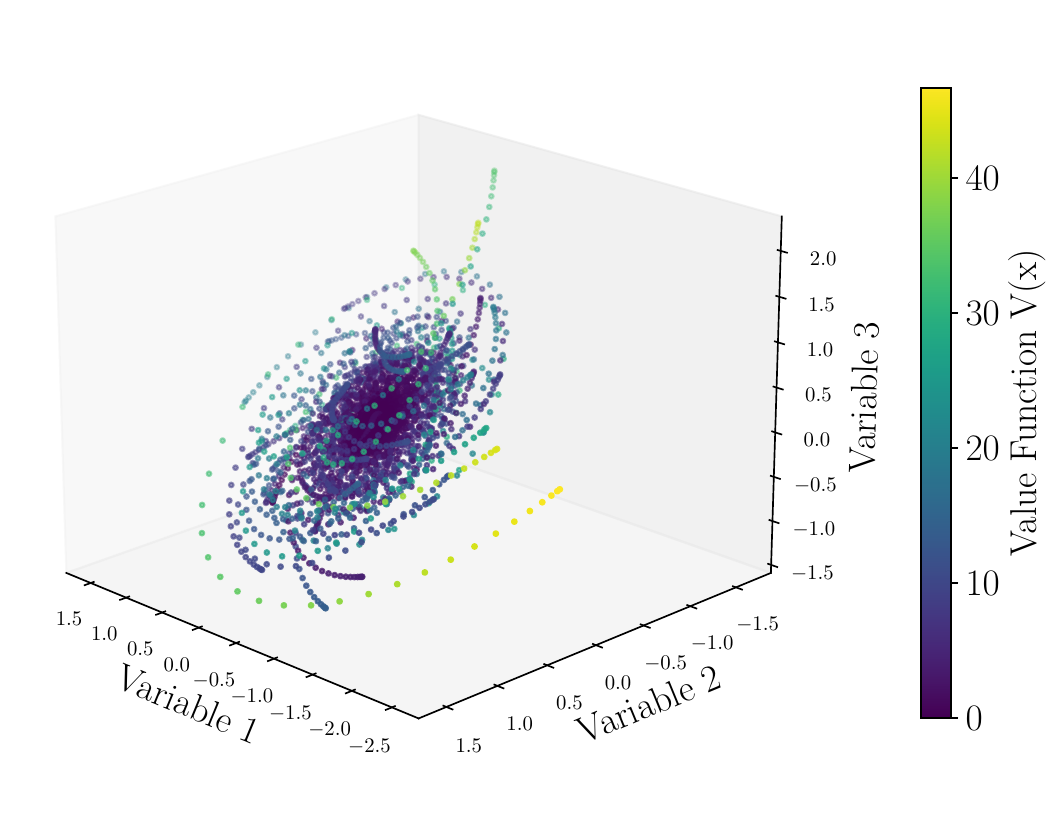}
    }

    \medskip

    \subfloat[3D POD Adjoint subspace]{%
        \includegraphics[width=0.49\linewidth]{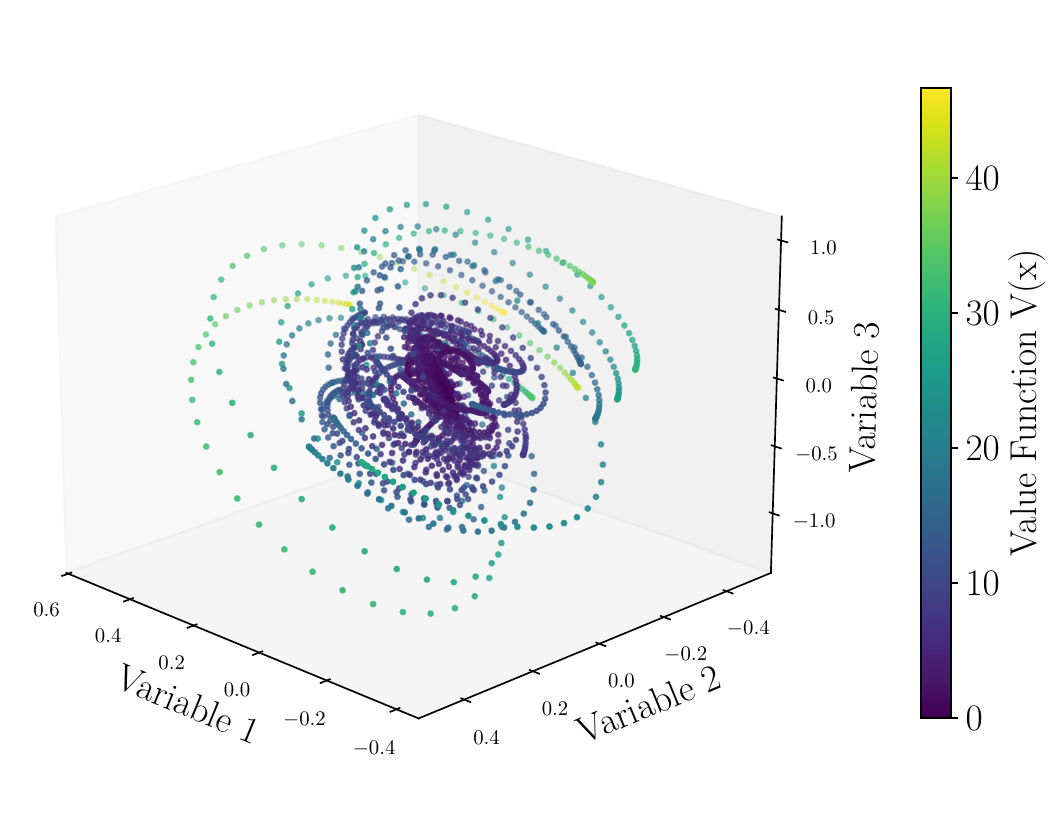}
    }
    \hfill
    \subfloat[3D BASOC subspace]{%
        \includegraphics[width=0.49\linewidth]{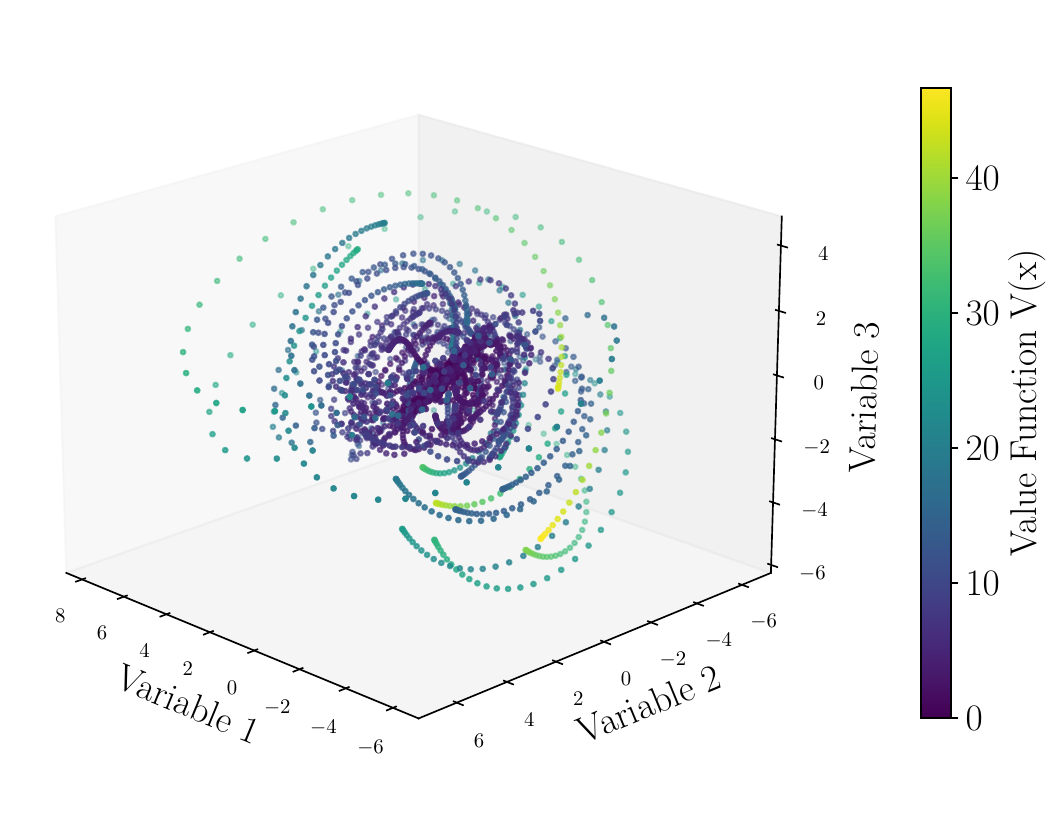}
    }
    \caption{Visualization of the value function for all trajectories on three-dimensional subspaces for the different methods. The colorbar indicates the value function \(V(x)\).}
    \label{fig:value_function_projection}
\end{figure}

\subsection{Full-state feedback control} 
\label{Full_state_feedback}

We evaluate surrogate full-state feedback control laws $u = \kappa_r(\Psi_r^T x)$ obtained by fitting $\kappa_r$ to the optimal trajectory data in different reduced coordinate systems $z = \Psi_r^T x$ obtained by BASOC, BT, POD, and POD Adjoint. To evaluate the different coordinates for learning near-optimal feedback laws, we test the surrogates as feedback laws for the full-order model with unseen initial conditions and compare the resulting costs to the (locally) optimal costs obtained by applying FBSM separately to each initial condition.
Specifically, we use the unseen initial conditions in \cref{subsec:toy_problem_compute_projections} for the nonlinear toy problem and in \cref{subsec:CGLE_compute_projections} for the Complex Ginzburg--Landau equation, with the corresponding locally optimal reference costs computed using FBSM. These initial conditions are not used for training, predictive evaluation, or hyperparameter selection. In order to assess whether dimensionality reduction improves or worsens the feedback law approximations, we also compare to surrogate models $u = \kappa(x)$ fit directly to the full-state optimal trajectory data.

The surrogate models $\kappa_r$ and $\kappa$ were fit using Gaussian process regression with a Matérn kernel having smoothness parameter $\nu_{\mathrm{GP}}=1.5$, scaled by a constant kernel, and an additive white-noise kernel. In our numerical experiments, we also tested Matérn kernels with smoothness parameters $\nu_{\mathrm{GP}}=0.5$ and $\nu_{\mathrm{GP}}=2.5$, but the results with $\nu_{\mathrm{GP}}=1.5$, provide a representative comparison among the methods.
We train the Gaussian-process surrogate using $2,400$ randomly-drawn samples from training trajectories, which are also used for kernel hyperparameter optimization using Scikit-learn \cite{scikit-learn}.
Alternatively, the training data points could be selected using the randomly pivoted Cholesky method \cite{chen2025randomly}.
We evaluate the predictions on $1,000$ held-out samples that are not used during fitting. Predictive accuracy is assessed using the coefficient of determination $(R^2)$ on both the training and held-out sets. The final control-performance results are tested using the unseen initial conditions described above.
\begin{figure}
    \centering
    \subfloat[]{%
        \includegraphics[width=0.45\linewidth]{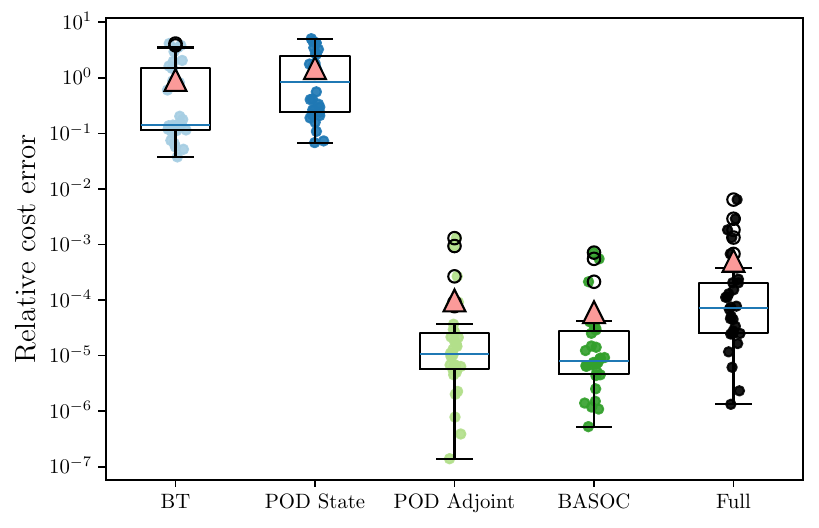}
    }
    \hfill
    \subfloat[]{%
        \includegraphics[width=0.45\linewidth]{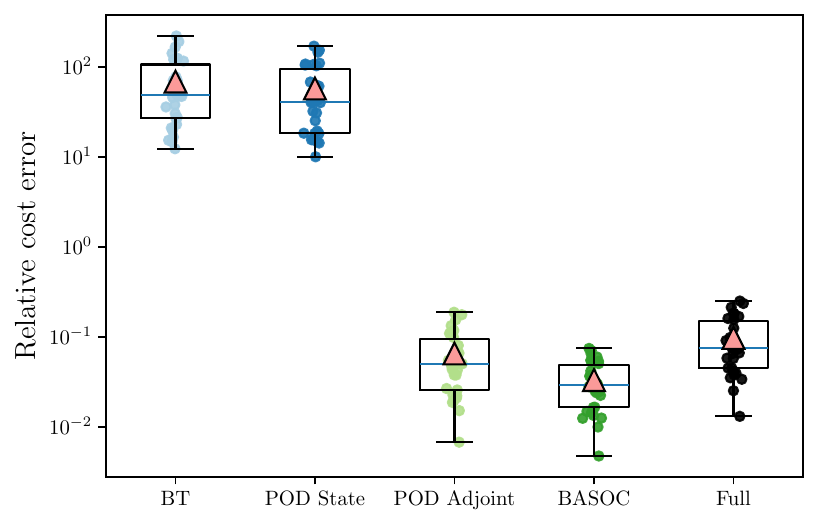}
    }
    \caption{Closed-loop full-state feedback using surrogate control laws for the challenging nonlinear toy model: (a) non-extrapolated test initial conditions; (b) extrapolated test initial conditions.}
    \label{fig:toy_model_full_state_feedback}
    \vspace{1em}
    \includegraphics[width=0.55\linewidth]{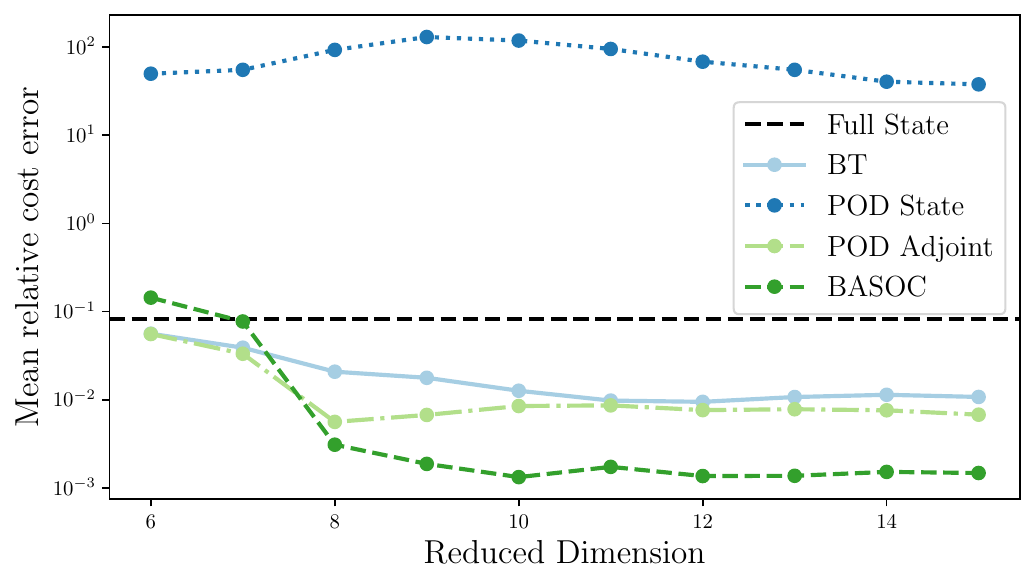}
    \caption{Mean relative control cost error for CGL equation full-state feedback, comparing BASOC with other methods. The dashed black line denotes the full-state surrogate baseline.}
    \label{fig:cgle_fullstate_feedback_cost}
\end{figure}

We quantify the surrogate control performance using the relative cost error $E=(J-J^*)/J^*$, where $J$ denotes the closed-loop cost achieved by the surrogate on the FOM and $J^*$ is the optimal cost obtained using trajectory optimization on the FOM. 
The results in Figure~\ref{fig:toy_model_full_state_feedback} are reported for non-extrapolated and extrapolated initial conditions for the toy model in Section~\ref{subsec:toy_problem_compute_projections}.
In both cases BT and state POD perform poorly. Remarkably, the surrogates fit using reduced $r=2$-dimensional BASOC and POD Adjoint coordinates outperform feedback surrogates fit using the full $3$-dimensional state.
In Figure~\ref{fig:cgle_fullstate_feedback_cost} for the CGL model, we also see that surrogate feedback laws fit in suitable reduced coordinates can dramatically outperform surrogate feedback laws fit using the full state.
The best-performing feedback surrogate, by approximately an order of magnitude in $E$, is obtained using BASOC coordinates over a wide range of reduced dimensions. 

\subsection{Open-loop ROM-based forecasting} To assess the quality of the ROMs in \eqref{eq:Reduced_order_model}, we forecast reduced-state trajectories $z(t)$ in an open-loop setting using optimal control inputs obtained via trajectory optimization of the FOM for test initial conditions. The forecasted reduced-state trajectory is then reconstructed via the trial basis $\Phi_r$ for comparison with optimal state trajectories from the FOM.

We compare $r=2$-dimensional Petrov-Galerkin ROMs of the toy model \eqref{eq:Toy_model_state_space} obtained by BASOC, POD, and POD Adjoint using the training ensemble of optimal trajectories, and BT about the equilibrium at the origin. 
We evaluate the ROMs on extrapolated and non-extrapolated initial conditions introduced in Section~\ref{subsec:toy_problem_compute_projections} for the toy model. 
Prediction performance with optimal control for a single trajectory and the normalized squared output prediction errors for each trajectory in the test ensemble are shown in Figure~\ref{fig:toy_model_rom_results}(a) and (b) respectively. 
In Figure~\ref{fig:toy_model_rom_results}, BASOC and POD Adjoint have comparable performance, making accurate output predictions with orders of magnitude lower normalized error than ROMs based on BT and POD. 
The accuracy of the POD Adjoint ROM was unexpected because the trial basis is not designed to approximate states.
The BT--Petrov--Galerkin model massively overpredicts the transient growth, while the POD-based model predicts no transient growth. In Figure~\ref{fig:CGL_rom_state_errors} we compare state forecasts by the different Petrov--Galerkin ROMs of the CGL equation \eqref{eq:CGLE_Equation} across several reduced dimensions. Normalized state forecast error is plotted for trajectories in the unseen testing ensemble. Since the nonlinearity in \eqref{eq:CGLE_Equation} is represented as a homogeneous polynomial, the reduced nonlinear term is obtained via its associated polarization map \cite{thomas2014polarization} as a third-order tensor. Because the prescribed open-loop input drives the FOM toward an unstable equilibrium, eventual divergence of the ROM forecast is expected. Accuracy should therefore be assessed by the duration of agreement with the FOM rather than asymptotic stability.

\begin{figure}[!htbp]
    \centering
    \subfloat[Output prediction (single trajectory)]{%
        \includegraphics[width=0.49\linewidth]{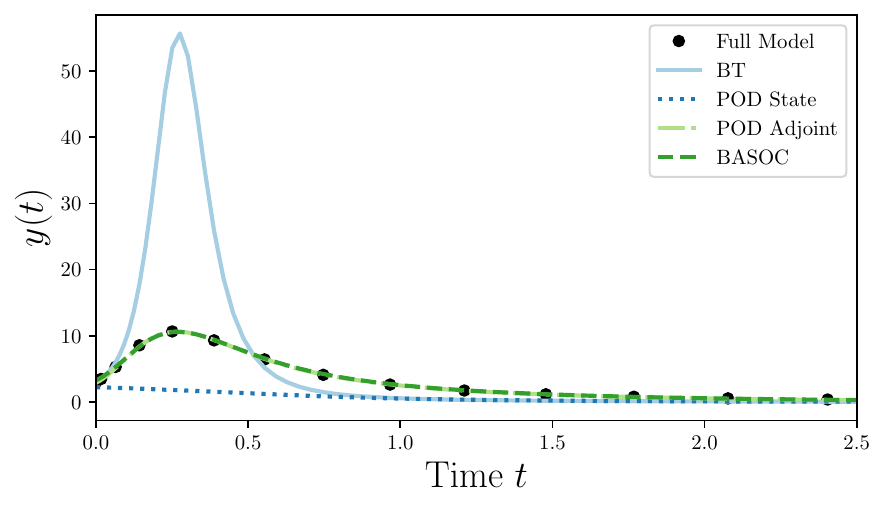}
    }
    \hfill
    \subfloat[Error on multiple test trajectories]{%
        \includegraphics[width=0.49\linewidth]{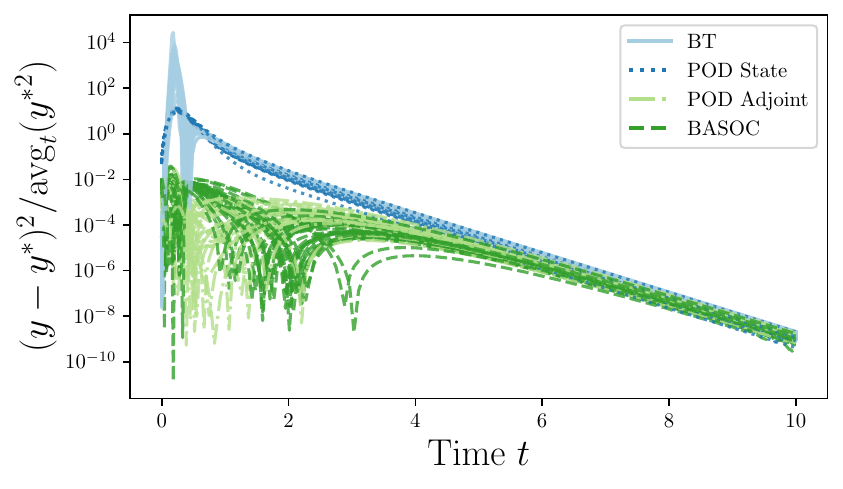}
    }
    \caption{ROM output prediction for the nonlinear toy model: (a) prediction for a single test trajectory;
    (b) normalized squared output prediction errors along optimal trajectories for extrapolated test initial conditions. Similar results are obtained for non-extrapolated initial conditions.}
    \label{fig:toy_model_rom_results}
\end{figure}

\begin{figure}[!htbp]
    \centering
    \includegraphics[width=0.80\linewidth]{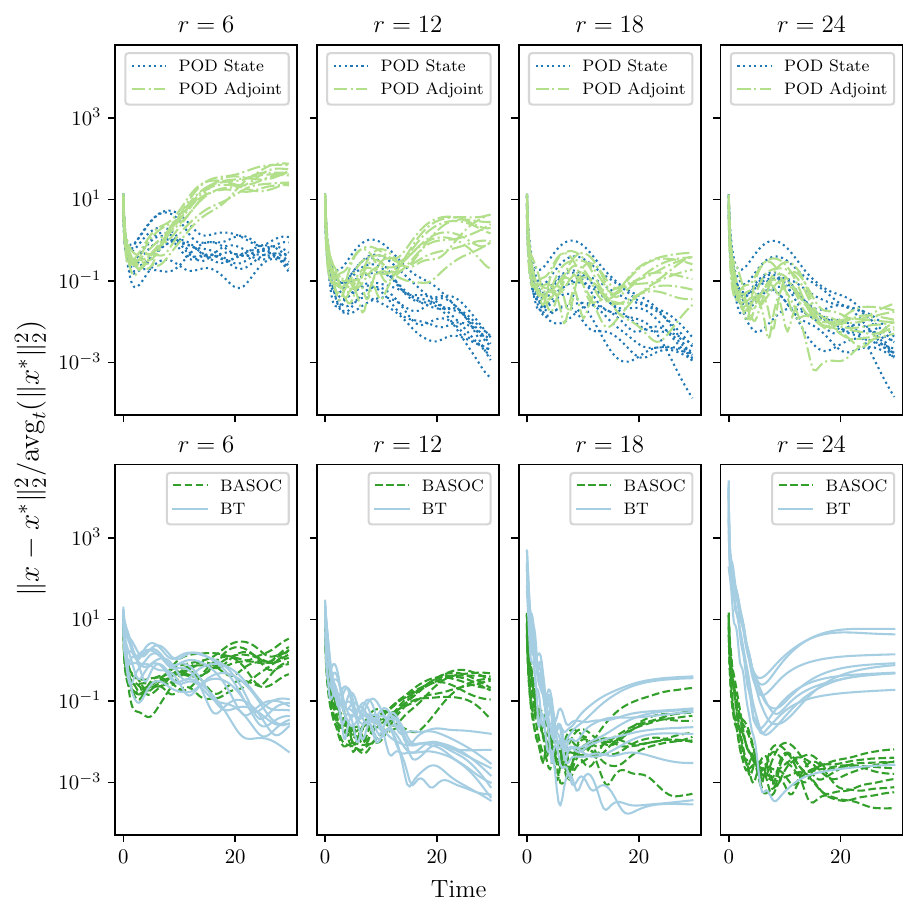}
    \caption{Normalized squared state-forecast error across optimal test trajectories for the nonlinear CGL equation for different reduced dimensions $r$.}
    \label{fig:CGL_rom_state_errors}
\end{figure}

Although BASOC is not uniformly best at lower reduced dimensions, its accuracy improves rapidly with dimension, achieving the strongest overall performance at the highest dimension tested.
 
\subsection{Open-loop state estimation with EKF-ROM}
\label{subsec:open_loop_state_estimation_with_ROM_EKF}
We evaluate different Petrov--Galerkin ROMs \eqref{eq:Reduced_order_model} within a reduced-order extended Kalman filter (EKF) to estimate the FOM's state based on measured outputs at discrete time steps $\Delta t = 0.01$. We evaluate state estimates along the testing ensemble of open-loop optimal trajectories before closing the loop in Section~\ref{subsec:Output_based_control_with_ROM_EKF}.
The EKF formulation and algorithm are provided in Appendix~\ref{app:state_estimation_control} and Algorithm~\ref{alg:output_feedback_rom_ekf}. For this experiment, we omit the control update step in Algorithm~\ref{alg:output_feedback_rom_ekf}.
Specifically, at the $k$th time step we apply the pre-computed open-loop optimal input $u_k = u_k^*$, measure the FOM's output $y_{k+1} = y_{k+1}^*$, and use it to correct the ROM's prediction at the $k+1$st time step.

\begin{figure}[t]
    \centering
    \subfloat[Toy model]{%
        \includegraphics[width=0.49\linewidth]{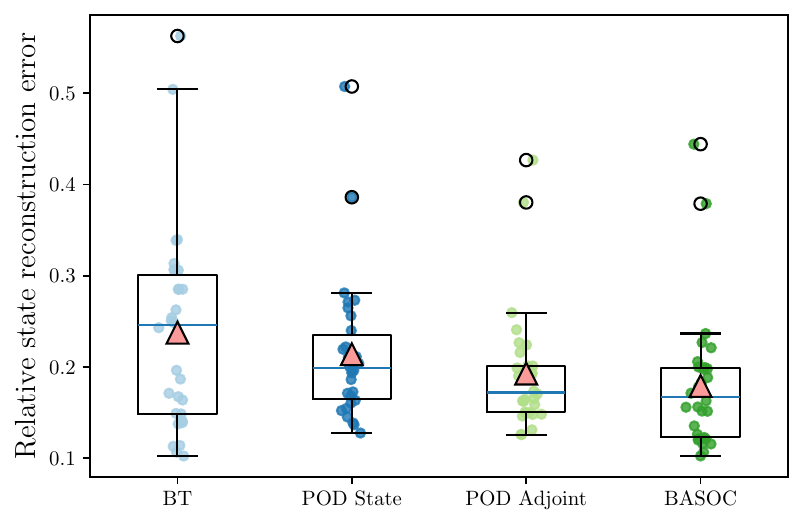}
    }
    \hfill
    \subfloat[CGL equation model]{%
        \includegraphics[width=0.49\linewidth]{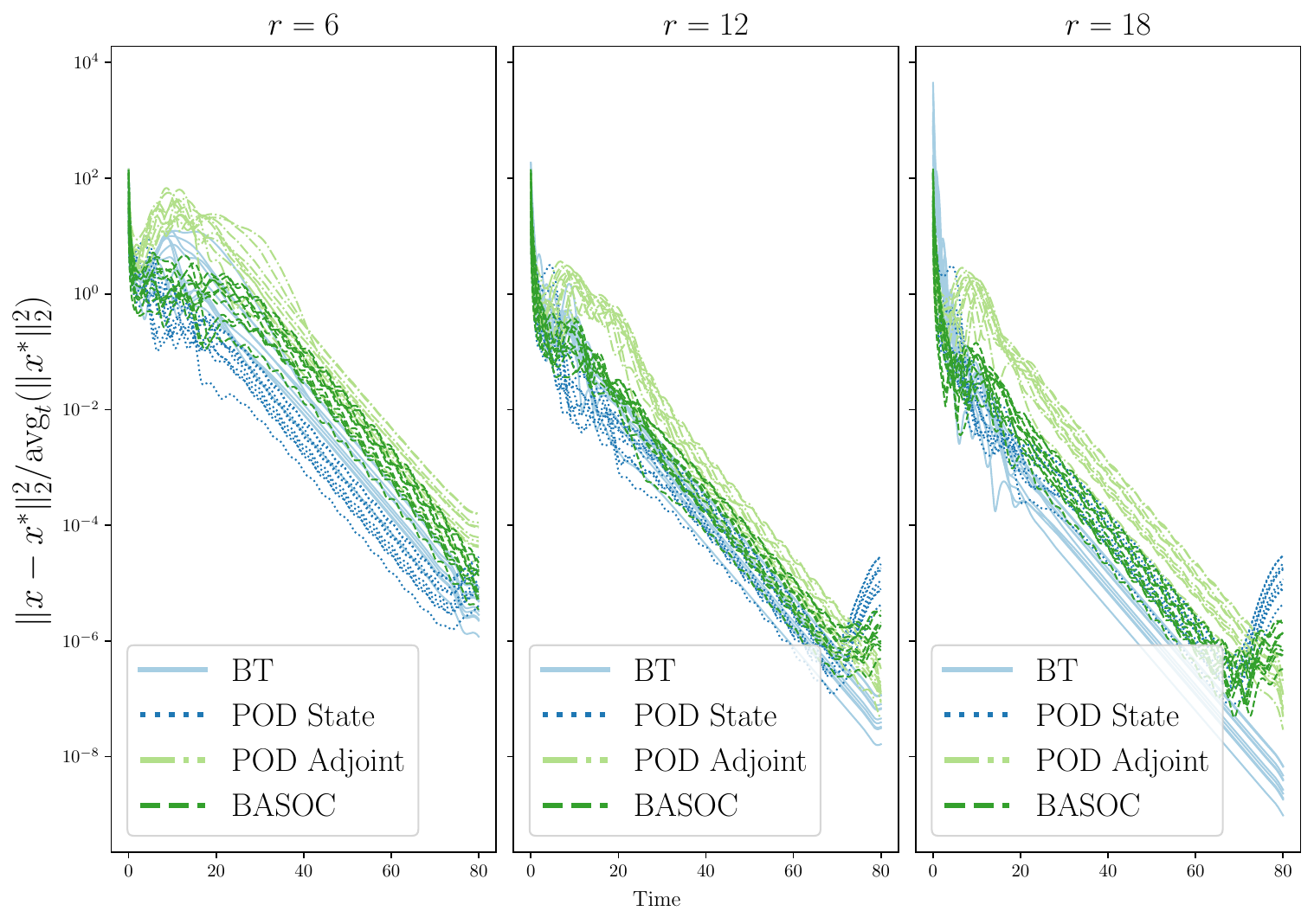}
    }
    \caption{Open-loop reduced-order state estimation using an EKF-ROM. (a) Challenging nonlinear toy model. (b) CGL equation model.}
    \label{fig:open_loop_ROM_EKF_state_estimation_results}
\end{figure}

For the three-dimensional toy model, the state estimation accuracy is evaluated on the same non-extrapolated and extrapolated initial conditions introduced in Section~\ref{subsec:toy_problem_compute_projections}. 
For each initial condition, the relative state-reconstruction error is computed as $\|X_{\mathrm{FOM}}^{*}-\Phi_r Z\|_{F}/\|X_{\mathrm{FOM}}^{*}\|_{F}$, where the columns of $X_{\mathrm{FOM}}^{*}$ are the state snapshots along the optimal FOM trajectory and the columns of $Z$ are the corresponding ROM state estimates. Figure~\ref{fig:open_loop_ROM_EKF_state_estimation_results}(a) shows the relative state-reconstruction error for the different ROMs with $r=2$. EKFs using state POD, POD Adjoint, and BASOC ROMs had similar performance with BASOC yielding slightly lower mean and median estimation errors. 
BT ROMs yielded a wider range of performance with higher mean and median error.
Similar results are obtained for extrapolated test initial conditions. 
We note that state POD and BT show substantial improvements within the EKF compared to open-loop ROM forecasting, likely due to output-based reduced-state corrections.

For the nonlinear CGL equation, Figure~\ref{fig:open_loop_ROM_EKF_state_estimation_results}(b) shows the normalized squared state-reconstruction error over time for test trajectories.
BASOC has excellent performance at early times (roughly $t < 10$) and improves as dimension is increased.
POD has similar performance at early times, with a few outlying trajectories with higher error.
At longer times --- where trajectories are closer to the equilibrium and we expect nonlinear effects to be less significant --- BT ROMs yield the best state estimates, with improvement as dimension is increased.
However, its early time performance can be much worse than BASOC. POD Adjoint has poor performance in the EKF.

\subsection{Output-based closed-loop feedback control via EKF-ROM}
\label{subsec:Output_based_control_with_ROM_EKF}
We now use the reduced-order state estimators validated in Section~\ref{subsec:open_loop_state_estimation_with_ROM_EKF} for output-based closed-loop feedback control of the FOM. The prediction, correction, and control-update steps are detailed in Appendix~\ref{app:state_estimation_control}, while Algorithm~\ref{alg:output_feedback_rom_ekf} details the closed-loop procedure.  We report results in terms of the relative cost error $E$ defined in Section~\ref{Full_state_feedback} with respect to the FBSM cost.

\begin{figure}[H]
    \centering
    \subfloat[Non-extrapolated initial conditions]{%
        \includegraphics[width=0.49\linewidth]{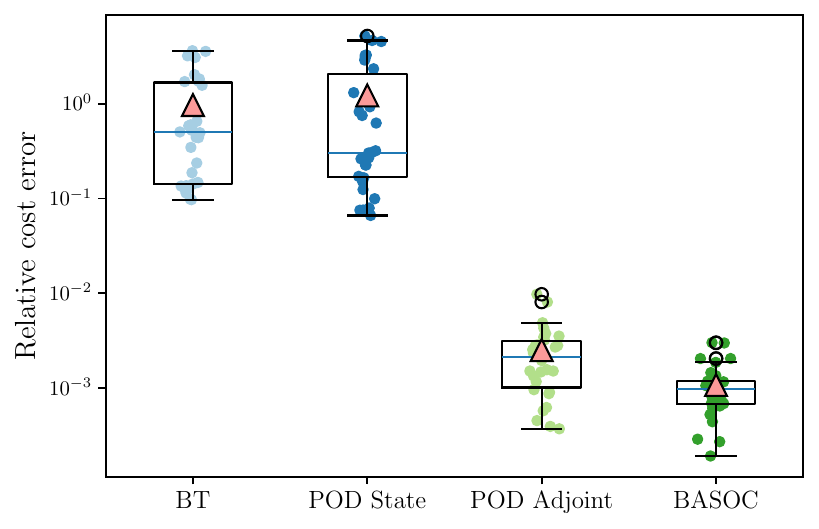}
    }
    \hfill
    \subfloat[Extrapolated initial conditions]{%
        \includegraphics[width=0.49\linewidth]{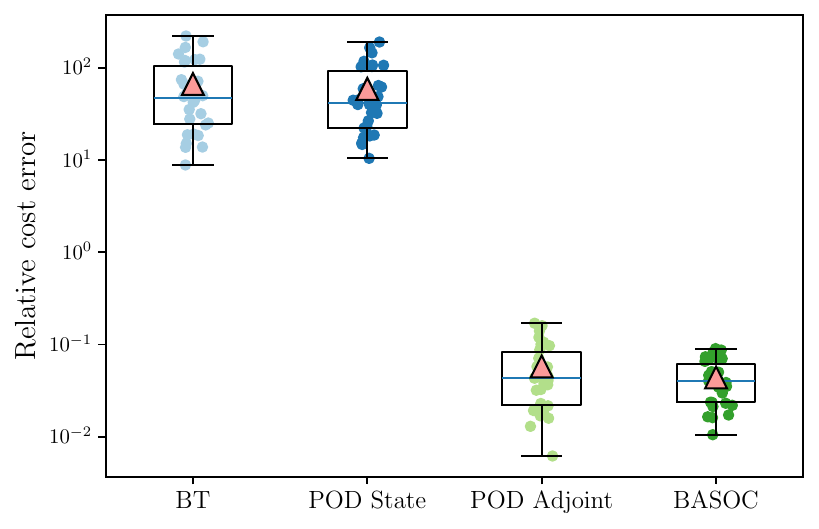}
    }
    \caption{Output-based closed-loop feedback control for the challenging nonlinear toy model for (a) non-extrapolated and (b) extrapolated test initial conditions.}
    \label{fig:output_based_closed_loop_feedback_toy_model}
\end{figure}

\begin{figure}[!htbp]
    \centering
    \includegraphics[width=0.55\linewidth]{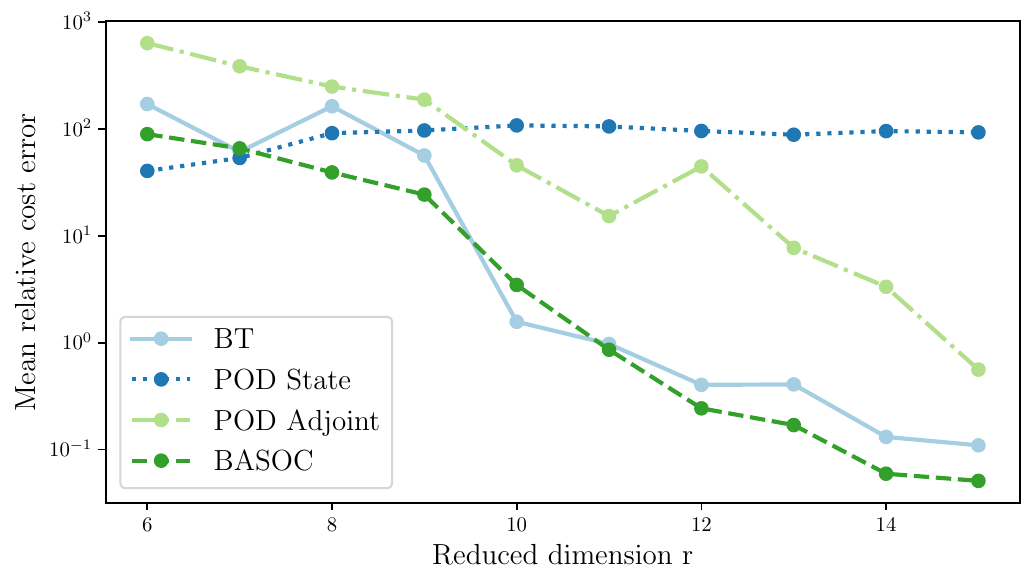}
    \caption{Output-based closed-loop feedback control using the EKF and surrogate control law for the CGL equation model.}
    \label{fig:output_based_closed_loop_control_CGLE}
\end{figure}

For the challenging nonlinear toy model, Figure~\ref{fig:output_based_closed_loop_feedback_toy_model} shows the results for output-based closed-loop feedback control. The same initial conditions introduced in Section~\ref{subsec:toy_problem_compute_projections} are used for this control task. Among the methods, BASOC achieves the lowest mean relative cost error, with the cost error computed relative to that of trajectory optimization. BT and state POD perform poorly on this problem.

For the nonlinear CGL equation, Figure~\ref{fig:output_based_closed_loop_control_CGLE} shows the results for output-based closed-loop feedback control using the different methods across reduced dimensions.
Across the reduced dimensions, the BASOC compensator yields either best or near best performance among the tested methods, with BT yielding comparable overall results. 
The poor performance of the POD-based compensator is likely due to its poorly-performing feedback surrogate (see Section~\ref{Full_state_feedback}), whereas the poor performance of POD Adjoint is likely due to its poorly-performing state estimator (see Section~\ref{subsec:open_loop_state_estimation_with_ROM_EKF}).

\section{Conclusion}
We have introduced Balanced Active Subspaces for Optimal Control (BASOC), a model reduction framework for the value function of high-dimensional nonlinear optimal control problems. BASOC yields an oblique projection that seeks to jointly capture directions in state space along which the value function is most sensitive and the distribution of states exhibits large variance. We explain several ways that this projection can be used for reduced-order modeling and control, and we evaluate some of these strategies on numerical examples.
For linear-quadratic optimal control problems we show that the BASOC becomes a snapshot-based approximation of initial condition balancing of the optimal closed-loop system.

We demonstrate the framework on a challenging nonlinear toy model and the nonlinear CGL equation, where we evaluate lower-dimensional approximations for the value function, optimal control law, dynamical system, and their combination into low-order compensators for output-based closed-loop control.
These numerical experiments show that optimal value functions and feedback laws are well approximated by BASOC reduced coordinates, whereas the other methods can yield substantially poorer approximations on the considered examples. 
Our experiments also show that POD on optimal adjoint snapshots --- without balancing against state variance --- can provide accurate approximations of the value function.
However, it may yield poor reduced-order models in much higher dimensions as the states generally do not lie in the span of the gradients.
Our results show that BASOC-based ROMs can provide accurate forecasts with improvements over balanced truncation and POD in nonlinear regimes.
While BASOC had competitive performance in all tasks, it showed the largest advantages of POD and balanced truncation on the three-dimensional toy model, which has strongly nonlinear dynamics specifically designed to defeat POD and balanced truncation.

Future work will extend BASOC to test for, and exploit balanced active subspaces in larger-scale nonlinear flow control problems. 
We will also leverage second-order value function information obtained from differential Riccati equations along optimal trajectories to identify active subspaces directly for feedback control laws, which depend on the gradient of the value function. 
Finally, we will seek to establish performance guarantees for BASOC-based ROMs and control laws.

\vspace{0.75em}
\noindent\textbf{Code availability:} Our code was written in Python and is available at \url{https://github.com/amehhh/BASOC}.

\appendix

\section{Forward--Backward Sweep Algorithm}
\begin{algorithm}[H]
\caption{Gradient-Based Trajectory Optimization (FBSM)}
\label{alg:Trajectory_optimization}
\begin{algorithmic}[1]
\STATE Initialize the weighting matrices $\widetilde{Q}$ and $R$
\STATE Choose $N$ quadrature nodes and a sufficiently large time horizon $\hat{T} \gg 1$
\STATE Obtain left Gauss-Radau quadrature nodes $\{t_i\}_{i=1}^N$ and weights $\{w_i\}_{i=1}^N$ on $[0,\hat{T}]$
\vspace{-1em}
\STATE Set $k=0$ and choose a stopping tolerance $\epsilon \ll 1$. Total number of $k$ iterations determined by optimizer.
\STATE Control roll-out: Choose an initial control trajectory $u^{(0)}(\cdot)$ and interpolate.
\STATE Specify the solver relative and absolute tolerances.
\WHILE{$\|\nabla J(u^{(k)}(\cdot))\| > \epsilon$}
\STATE Forward sweep: Simulate the system dynamics in \eqref{eq:forward_dynamics_state_equation} from $x_{0}$ and interpolate the solution to obtain state trajectory $x^{(k)}(\cdot)$.
\STATE Compute the cost functional $J(u^{(k)}(\cdot))$ in \eqref{eq:Quadratic_cost_function} approximated with quadrature.
\vspace{-1em}
\STATE Backward sweep: Simulate the adjoint dynamics in \eqref{eq:backward_dynamics_adjoint_equation} backward in time with terminal condition $\lambda(\hat{T})=0$ to obtain the adjoint trajectory $\lambda^{(k)}(\cdot)$
\STATE Compute the gradient of the cost functional in \eqref{eq:closed_form_gradient} $\nabla_u J(u^{(k)}(\cdot))$ with respect to the control parameters via quadrature approximation. 
\STATE Optimization update step: Pass $J(u^{(k)}(\cdot))$ and $\nabla_u J(u^{(k)}(\cdot))$ to an optimizer (e.g., BFGS or L-BFGS) to obtain the updated control $u^{(k+1)}(\cdot)$.
\STATE Update $k \gets k+1 $
\ENDWHILE
\STATE Set $u^*(\cdot)$ to the converged control trajectory.
\STATE Solve the forward and backward dynamics using $u^*(\cdot)$ to obtain $x^*(\cdot)$ and $\lambda^*(\cdot)$.
\STATE Compute the optimal cost $J^*$ along the optimal trajectory using \eqref{eq:Quadratic_cost_function}.
\STATE Compute $V(\cdot)$ along the optimal trajectory using \eqref{eq:value_function_backward_quadrature}.
\RETURN Optimal cost $J^*$, value function $V(\cdot)$, optimal control $u^*(\cdot)$, optimal state $x^*(\cdot)$, and optimal adjoint $\lambda^*(\cdot)$.
\end{algorithmic}
\end{algorithm}

For the optimization problem in the nonlinear case, we obtain a closed-form expression for the gradient of the cost functional with respect to the optimization parameters using an adjoint formulation \cite{giles2000adjoint}. With the adjoint equation in \eqref{eq:backward_dynamics_adjoint_equation} and the quadratic cost functional as defined in \eqref{eq:Lagrange_cost_functional}, \eqref{eq:Quadratic_cost_function} the closed form is given as:
\begin{equation}
% DJ(u(\cdot))[v(\cdot)] = \int_{0}^{\infty}\!\!\!\!\!\!\!\!\!
% \underbrace{\left( \nabla_u J(u(t)) \right)^{\top}}_{\displaystyle \nabla_u L(x(t),u(t)) + B^{\top}\lambda(t)} \!\!\!\!\!\!\!\!\!\!v(t)\, dt.
DJ(u(\cdot))[v(\cdot)] = \int_{0}^{\infty}
\underbrace{\left[\nabla_u L(x(t),u(t)) + B^{\top}\lambda(t)\right]}_{\left( \nabla_u J(u(t)) \right)^{\top}} v(t)\, dt.
\label{eq:closed_form_gradient}
\end{equation}

% DJ(u(\cdot))[v(\cdot)]
% =
% \left\langle \nabla_u J(u(\cdot)),\,v(\cdot)\right\rangle
We employ left Gauss-Radau quadrature with a linearly transformed domain to approximate the continuous integrals as suggested in \cite{ross2012pseudospectral}. Its inclusion of left end-point node on the discretization domain enforces the initial condition for infinite-horizon optimal control problems. 
% See \cite{garg2010unified} for accuracy and convergence properties of high-order Gaussian quadrature transcription methods in optimal control.
High-order gaussian quadrature also provides an accurate integration with fewer nodes, reducing the computational expense of the optimization problem and Monte-Carlo estimation of the covariance matrices. We employ Barycentric interpolation due to its stability and compatibility with the Legendre polynomial framework of Gaussian quadrature methods \cite{Wang2012}. In our numerical experiments, we employ the Scipy \cite{Virtanen2020} BFGS \cite{Nocedal2006} quasi-newton optimizer with a tolerance of $10^{-6}$ and gradient tolerance of $10^{-8}$. The optimization stops when a convergence criterion is met. 
 \begin{figure}[t]
    \centering
    \subfloat[Optimization convergence (Toy model)]{%
        \includegraphics[width=0.49\linewidth]{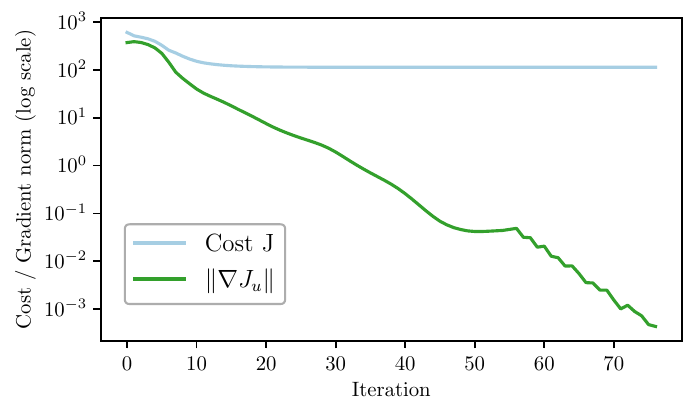}
    }
    \hfill
    \subfloat[Optimization convergence (CGL model)]{%
        \includegraphics[width=0.49\linewidth]{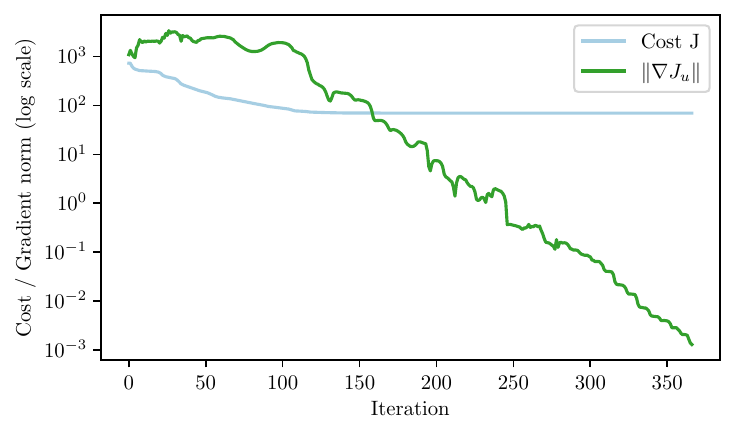}
    }
    \caption{Convergence of the cost function and its gradient norm during trajectory optimization for (a) the nonlinear toy model and (b) the CGL model.}
    \label{fig:Trajectory_optimization_results}
\end{figure}
Our implementation of trajectory optimization, whose convergence on representative trajectories is plotted in Figure~\ref{fig:Trajectory_optimization_results}, required $\approx1$ minute per trajectory for the nonlinear toy model and $\approx15$ minutes per trajectory for the nonlinear CGL model. Optimization over multiple trajectories is embarrassingly parallel. All numerical experiments were conducted on a MacBook Pro equipped with an Apple M4 Pro chip and 24 GB of unified memory. The value function $V(x)$ at any node $t_i$ along the optimal trajectory is computed with quadrature as:
\begin{equation}
V(x(t_i))
\approx
\sum_{s=i}^{N-1} w_s\,L\left(x^*(t_s),u^*(t_s)\right).
\label{eq:value_function_backward_quadrature}
\end{equation}

\section{State Estimation and Output-Based Closed-Loop Control}
\label{app:state_estimation_control}
\vspace{0.1em}

% For both the open-loop state-estimation and output-based closed-loop control experiments, the EKF is updated at discrete times \(t_k=k\Delta t\), with \(\Delta t=\text{[insert numerical value]}\). In the output-based feedback experiments, the control input is held constant over each interval \([t_k,t_{k+1})\), i.e., \(u(t)=u_k\), and is updated after incorporating the measurement at \(t_{k+1}\). The same piecewise-constant control is used when propagating the FOM and ROM dynamics.
This appendix presents the EKF-ROM formulation and algorithm used for the open-loop state-estimation experiment in Section~\ref{subsec:open_loop_state_estimation_with_ROM_EKF} and the output-based closed-loop control experiment in Section~\ref{subsec:Output_based_control_with_ROM_EKF}. The same estimator is used in both settings, with the open-loop distinction specified in its respective section. 

In both experiments, we use a fixed time step of $\Delta t=0.01$ for the discrete-time EKF. In particular, we do the following : Let $F_r^{\Delta t}(z,u) = z(t+\Delta t), \quad z\in\mathbb{R}^r$ denote the discrete-time flow map of the ROM \eqref{eq:Reduced_order_model} with a constant control input $u \in \mathbb{R}^m$ over a time interval of length $\Delta t$. We approximate it using SciPy’s RK45 integrator (version 1.16.2) \cite{Virtanen2020}. In the prediction step, the a priori state is obtained by propagating the a posteriori state estimate via the ROM dynamics as:
\begin{equation}
\hat{z}_{k+1}^{-} = F_r^{\Delta t}(\hat{z}^+_k, u_k).
\label{eq:propagte_rom_dynamics}
\end{equation}
To propagate the a posteriori state-estimation error covariance matrix $\hat{P}_k^{+} \in \mathbb{R}^{r \times r }$, let: 
\begin{equation}
    J_k^z = \left. \frac{\partial F_r^{\Delta t}(z,u)}{\partial z} \right|_{(\hat{z}^+_k,u_k)},
    \label{eq:ROM_Jacobian}
\end{equation}
denote the Jacobian of the ROM discrete-time flow map evaluated at $(\hat{z}_k^+,u_k)$. Using a Euler forward discretization for propagating the ROM dynamics and a sufficiently small $\Delta t$, we obtain a first-order approximation of the discrete-time flow map Jacobian as :
\begin{equation}
J_k^z \approx I_r + \Delta t \left. \frac{\partial f_r(z,u)}{\partial z} \right|_{(\hat{z}^+_k,u_k)}.
\label{ROM_flow_Jacobian(Euler)}
\end{equation}
Let $h_r:\mathbb{R}^r \rightarrow \mathbb{R}^p$ denote the linear ROM output function. Its Jacobian is 
\begin{equation}
H = \frac{\partial h_r(z)}{\partial z} \in \mathbb{R}^{p\times r}
\label{eq:Jacobian_ROM_output_function}
\end{equation}
We assume that the process and measurement noises are
uncorrelated, zero-mean Gaussian random processes with covariance matrices $\hat{Q}$ and $\hat{R}$, respectively,
following \cite{Schneider2013}. In addition, the initial reduced state estimate is set to $\hat{z}_0^+ = \Psi_r^T x_0$, while the estimation error covariance matrix $\hat{P}_0^+$ is initialized following \cite{Schneider2013}. The error covariance update is obtained using the Joseph stabilized form to ensure it preserves symmetry and positive semidefiniteness \cite{Schneider2013}. See \cite{Farrell2001ReducedKF} for reduced state estimation of linear systems. 

\begin{algorithm}[H]
\caption{Output-based closed-loop feedback control via EKF-ROM}
\label{alg:output_feedback_rom_ekf}
\begin{algorithmic}[1]
\STATE Choose the process-noise covariance matrix
$\hat{Q} \in \mathbb{R}^{r \times r}$ and measurement noise covariance $\hat{R} \in \mathbb{R}^{p \times p}$.
\STATE Choose a small time step $\Delta t=0.01$, define number of time steps as $n_t=\hat{T}/\Delta t$.
\STATE Initialize the a posteriori state-estimation error covariance matrix as $\hat{P}_{0}^{+}\in\mathbb{R}^{r\times r}$.
\vspace{-1.2em}
\STATE Initialize a posteriori state $\hat{z}_0^+\in\mathbb{R}^r$ and initial control $u_0=\hat{u}^{*}(\hat{z}_0^+)$.
\FOR{$k=0,\ldots,n_t-1$}
    \STATE Propagate the FOM dynamics from time step $k$ to $k+1$ over the small time interval $\Delta t$ using $u_k$ and obtain the FOM measurement $y_{k+1} \in \mathbb{R}^p$.
    \STATE Propagate the a posteriori reduced state estimate through the ROM dynamics from time step $k$ to $k+1$ over the small time interval $\Delta t$ using $u_k$ to obtain the a priori state estimate $\hat{z}_{k+1}^{-}$.
    \STATE Propagate the a posteriori state-estimation error covariance over the small time interval to obtain the a priori covariance: $\hat{P}_{k+1}^{-} = J_k^{z} \hat{P}_k^{+} (J_k^{z})^{T} + \hat{Q}$.
    \STATE Compute the innovation covariance as $S_{k+1}=H\hat{P}_{k+1}^{-}H^{T}+\hat{R}$.
    \STATE Compute the Kalman gain $K_{k+1} \in \mathbb{R}^{r\times p} $ as $K_{k+1}= \hat{P}_{k+1}^{-}H^{T}S_{k+1}^{-1}$.
    \STATE Compute the innovation as $e_{k+1} = y_{k+1} - H\hat{z}_{k+1}^{-}$.
    \STATE Update the state estimate as $\hat{z}_{k+1}^+ = \hat{z}_{k+1}^{-} + K_{k+1} e_{k+1}$.
    \STATE Update the a priori state-estimation error covariance matrix to obtain the a posteriori covariance matrix using the Joseph stabilized form: $ \hat{P}_{k+1}^{+} = (I_r-K_{k+1}H)\hat{P}_{k+1}^{-}(I_r-K_{k+1}H)^{T} + K_{k+1}\hat{R}K_{k+1}^{T}$.
    \STATE Update the control input with near-optimal control from surrogate control law evaluated at the a posteriori state estimate $u_{k+1} = \hat{u}^*\!\left(\hat{z}_{k+1}^{+}\right)$.
\ENDFOR
\end{algorithmic}
\end{algorithm}

\bibliographystyle{siamplain}
\bibliography{references}
\end{document}